\documentclass[12pt, a4paper, parskip=half, abstracton, bibliography=totoc]{scrartcl}
\pdfoutput=1

\usepackage{array}
\usepackage{marginnote}
\usepackage{xcolor}
\usepackage{amscd,amssymb,amsfonts,amsmath,latexsym,amsthm}
\usepackage{hyperref}
\usepackage[all,cmtip]{xy}
\usepackage{mathrsfs}
\usepackage{graphicx}
\usepackage{bm} 
\usepackage{mathtools} 

\usepackage{etoolbox}

\let\counterwithout\relax

\usepackage{chngcntr} 

\usepackage{defs_pp1}
\usepackage{slashed}
\usepackage[utf8]{inputenc}
\usepackage{microtype}
\usepackage[english]{babel}

\title{Coarse cohomology theories}
\author{
Ulrich Bunke\thanks{Fakult{\"a}t f{\"u}r Mathematik,
Universit{\"a}t Regensburg,
93040 Regensburg,
GERMANY\newline
ulrich.bunke@mathematik.uni-regensburg.de} 
\and
Alexander Engel\thanks{Institut f{\"u}r Mathematik und Informatik der Universit{\"a}t Greifswald\newline
Walther-Rathenau-Str.\,47 in 17489 Greifswald\newline
alexander.engel@uni-greifswald.de}
}

\numberwithin{equation}{section}
\counterwithout{footnote}{section}

\newtheorem{theorem}{Theorem}[section] 
\newtheorem{prop}[theorem]{Proposition}
\newtheorem{lem}[theorem]{Lemma}

\newtheorem{ddd}[theorem]{Definition}
\newtheorem{kor}[theorem]{Corollary}

\theoremstyle{remark}
\theoremstyle{definition}
\newtheorem{ex-alt}[theorem]{Example}
\newtheorem{rem-alt}[theorem]{Remark}

\newenvironment{ex}
{%
	\pushQED{\qed}\begin{ex-alt}}
	{\popQED\end{ex-alt}}

\newenvironment{rem}
{%
	\pushQED{\qed}\begin{rem-alt}}
	{\popQED\end{rem-alt}}

\newcommand{\deff}{\mathrm{def}}
\newcommand{\wfl}{\mathrm{wfl}}

\newcommand{\All}{\mathbf{All}}

\newcommand{\UBC}{\mathbf{UBC}}

\newcommand{\Rips}{\mathrm{Rips}}
\newcommand{\Coind}{\mathrm{Coind}}

\newcommand{\Yo}{\mathrm{Yo}}

\newcommand{\Res}{\mathrm{Res}}

\newcommand{\Orb}{\mathbf{Orb}}

\newcommand{\BC}{\mathbf{BornCoarse}}

\newcommand{\Fin}{\mathbf{Fin}}

\newcommand{\Cofib}{\mathrm{Cofib}}

\newcommand{\Fib}{{\mathrm{Fib}}}

\newcommand{\bF}{{\mathbf{F}}}

\newcommand{\cO}{{\mathcal{O}}}

\newcommand{\cY}{{\mathcal{Y}}}

 \newcommand{\Cone}{{\mathtt{Cone}}}

\newcommand{\bP}{\mathbf{P}}

\newcommand{\HX}{H\!\mathcal{X}}
\newcommand{\CX}{C\!\cX}

\newcommand{\Spc}{\mathbf{Spc}}

\newcommand{\IZ}{\mathbb{Z}}

\newcommand{\HAX}{H\! A\cX}

\newcommand{\hlg}{\mathrm{hlg}}
\newcommand{\free}{\mathrm{free}}
\newcommand{\op}{\mathrm{op}}
\newcommand{\BG}{\mathbf{B}G}

\begin{document}

\maketitle

\begin{abstract}
We propose the notion of a coarse cohomology theory and study the examples of coarse ordinary cohomology, coarse stable cohomotopy and of coarse cohomology theories obtained by dualizing coarse homology theories.

We show that the dualizing spectrum of a finitely generated torsion-free group only depends on the coarse motivic spectrum represented by the underlying bornological coarse space of the group. This in particular implies a conjecture  of J.\,R.\,Klein that the dualizing spectrum of a group is a coarse invariant.

\end{abstract}

\tableofcontents

\section{Introduction}

In \cite{buen} we have introduced the category $\BC$ of bornological coarse spaces and the notion of a $\bC$-valued coarse homology theory  $E\colon \BC\to \bC$, where $\bC$ is a   stable  $\infty$-category.   In the   present paper we place the notion of a coarse cohomology theory into the same framework.

Let  $\bD$ be a  stable  $\infty$-category and $E\colon \BC^{\op}\to \bD$ be a functor.
\begin{ddd}\label{weijgowefwerfer}
We call $E$  a coarse cohomology theory if $E^{\op} \colon \BC\to \bD^{\op}$ is a coarse homology theory. 
\end{ddd}

The main purpose of the present paper is to present the construction of the following   three examples of coarse cohomology theories:
\begin{enumerate}
\item To every abelian group $A$ we associate the coarse ordinary cohomology theory
$$\HAX\colon \BC^{\op}\to \Ch_{\infty}\, .$$
Our Definition \ref{ewkfhweiofwefewf} extends the original definition of Roe \cite{roe_coarse_cohomology} to the context of bornological coarse spaces, see {Lemma~\ref{lem23435tzrgewrt}.} 
 \item  \label{ergiooergregeg1}  If $\bC$   is a   bicomplete stable $\infty$-category tensored over $\Sp$, then in Definition~\ref{wrowpt32432424234234} we associate  to every object
 $C$ of $\bC$   a  $\bC$-valued coarse cohomology theory $Q_{C}$. For the category of spectra $\bC=\Sp$ and  for the sphere spectrum $C=S$ we obtain a coarse version $Q_{S}$ of stable cohomotopy.
  \item \label{ergiooergregeg} If $E$ is a $\bC$-valued coarse homology theory and $C$ is an object of $\bC$, then in the Definition \ref{rgoirjgoiergergreg} we define the dual $\Sp$-valued coarse cohomology theory $D_{C}(E)$ by forming the mapping spectrum with target $C$. We also consider versions of this construction where we replace the mapping space functor by some internal mapping object functor or a suitable power functor.  
\end{enumerate}

{Note that there are also other approaches to a general framework for coarse cohomology theories. Let us mention exemplary the work of Schmidt \cite{schmidt},  Hartmann \cite{hartmann}, and Wulff \cite{Wulff:2020vh}.}

In Section \ref{regoiegpergergreg} we provide an application of the coarse cohomology $Q_{S}$ discussed above in  Item \ref{ergiooergregeg1}.  
In order to formulate our result we recall the following:

\begin{enumerate}
\item Let $G$ be a discrete   group and define   the dualizing spectrum $\underline{D}_{G} \coloneqq S[G]^{G}$.  Here  $S[G]$ is the suspension spectrum associated to the  underlying  discrete  space of $G$. The right-action of $G$ on itself induces an action   on $S[G]$, and $S[G]^{G}$ denotes the spectrum of   fixed points.

\item In \cite[Def.\,4.3]{buen} we constructed a universal coarse homology theory $$\Yo^{s}\colon\BC\to \Sp\cX$$ with values in the stable $\infty$-category  of coarse motivic spectra.
\item By   \cite[Ex.\,2.21]{buen}  the  group $G$ gives naturally rise to a bornological coarse space $G_{can,min}$  in $\BC$  and therefore to a coarse motivic spectrum $\Yo^{s}(G_{can,min})$ in $\Sp\cX$. Note that $\Yo^{s}(G_{can,min})$ is in particular an invariant of the quasi-isometry class of $G$.
\end{enumerate}
 \begin{theorem}[Theorem \ref{ergoiergregege}]
If $G$ is  {finitely generated and} torsion-free, then we have an equivalence
$\underline{D}_{G}\simeq Q_{S}(G_{can,min})$ in $\Sp$. 
\end{theorem}

 The following  corollary settles a generalization of a conjecture stated by Klein \cite[Conj.~on Page~455]{klein}.
\begin{kor}[Corollary~\ref{rwfgiowjgogrgergeg}]\label{thm243erdwe}
 For a finitely generated and  torsion-free group $G$
the spectrum $\underline{D}_{G}$ only depends on the coarse motivic spectrum $\Yo^{s}(G_{can,min})$.
In particular, it is an invariant of the quasi-isometry class.
\end{kor}


\begin{rem} 
The $\Z$-linear analogues of the above results have been settled   by Roe as we explain in the following.

Recall that a group $G$ 
is of the type $\mathit{FP}$ if  the trivial $G$-module $\Z$ admits a finite-length projective resolution by finitely generated $\Z[G]$-modules. We let $n$ be the cohomological dimension of $G$ and   define the dualizing $G$-module by $D_{G}^{\Z,n} \coloneqq H^n(G;\IZ[G])$. 
Here we consider $\Z[G]$ as a $G$-module with the action induced by the right multiplication on $G$, and the $G$-action on the cohomology is induced from the left $G$-action on $\Z[G]$.  

By $\underline{D}_{G}^{\Z} \coloneqq  H(G;\IZ[G])$ we denote the cohomology complex
of $G$ with coefficients in the module $\Z[G]$. We consider it as an object of the stable $\infty$-category $\Ch_{\infty}$ of chain complexes in $\Ab$ with quasi-isomorphisms inverted. 
The chain complex $\underline{D}_{G}^{\Z}$ is the $\Z$-linear analogue of $\underline{D}_{G}$ defined above. 

 The group $G$ is  called a Bieri--Eckmann duality group, if it is of type $\mathit{FP}$ and    $\underline{D}_{G}^{\Z}\simeq   D_{G}^{\Z,n}[-n]$ (with $G$-action forgotten), i.e., $\underline{D}_{G}^{\Z}$ has only one non-trivial cohomology group which sits in degree $n$ and is isomorphic to $D_{G}^{\Z,n}$. If in addition 
the underlying abelian group of  $D_{G}^{\Z,n}$ is torsion-free,  then the four assertions stated in \cite[Thm.\,VIII.10.1]{brown} are satisfied. In particular, the group~$G$ satisfies a version of Poincar\'{e} duality in the sense that there are natural isomorphisms $H^i(G;-) \cong H_{n-i}(G;D_G^{\Z,n} \otimes -)$ of functors on the category of $G$-modules. 

 A Bieri--Eckmann duality group with the property 
  that $D_G^{\Z,n}$ is torsion-free  is called a Poincar\'{e} duality group.

It is known from the work of Roe \cite[Prop.\,2.6]{roe_index_coarse} (see also \cite[Thm.\,8]{gersten}) that the cohomology groups  $H^{*}(G;\IZ[G])$  are invariants of $\Yo^{s}(G_{can,min})$    because they coincide with the coarse ordinary cohomology groups of $G_{can,min}$. Thus if $G$ is of type FP, then the property of being a Bieri--Eckmann  or being a  Poincar\'{e} duality group
only depends on the equivalence class of $\Yo^{s}(G_{can,min})$.
 \end{rem}

\paragraph{Acknowledgements}
\textit{The authors were supported by the SFB 1085 ``Higher Invariants'' funded by the Deutsche Forschungsgemeinschaft DFG.}

\section{Coarse cohomology theories}
\label{sec5645356435}

In Section \ref{sec2345t2345} we give an axiomatic  definition of the notion of a coarse cohomology theory which  by the Corollary \ref{kor2435tre} it is eventually equivalent to Definition \ref{weijgowefwerfer}. 
In Section \ref{sec124354trew} we construct first examples by dualizing coarse homology theories.
  In Section \ref{wgowriepgwegweere} we discuss coassembly maps for coarse cohomology theories. 

\subsection{{Definition and basic properties}}
\label{sec2345t2345}

In this section we define coarse cohomology theories by dualizing the axioms for coarse homology theories given in \cite[Sec.\,4.4]{buen}. 

We will use notions for bornological coarse spaces and coarse homology theories introduced in \cite{buen}. In particular, we  assume familiarity of the reader with the material in \cite[Sec.\,2--4]{buen}.

 Let $\bC$ be a complete stable $\infty$-category.

We equip the two-point set $\{0,1\}$ with the maximal bornological and coarse structures.
The category $\BC$ has a symmetric monoidal structure $\otimes$ introduced in  \cite[Ex.\,2.32]{buen}. 
Since $\{0,1\}$ is bounded, 
for every bornological coarse space $X$ the projection $\{0,1\}\otimes X\to X$ is a morphism of bornological coarse spaces.

We consider a functor $E\colon \BC^{\op}\to \bC$.
\begin{ddd}
$E$ is called coarsely invariant if for every bornological coarse space $X$ the projection $\{0,1\}\otimes X\to X$ induces an equivalence
$E(X)\to E(\{0,1\}\otimes X)$.
\end{ddd}

Let $\cY:=(Y_{i})_{i\in I}$ be a big family on a bornological coarse space $X$ \cite[Def.\,3.2]{buen}. Set
\[E(\cY) \coloneqq \lim_{i\in I} E(Y_{i})\]
and note that have a natural morphism
\[E(X)\to E(\cY)\,.\]
The definition of excisiveness involves the notion of complementary pairs \cite[Def.\,3.5]{buen}.
\begin{ddd}
$E$ is called excisive if for any complementary pair $(\cY,Z)$   on a bornological coarse space $X$ the square
$$\xymatrix{E(X)\ar[r]\ar[d]&E(Z)\ar[d]\\E(\cY)\ar[r]&E(Z\cap \cY)}$$
is cartesian.
\end{ddd}

Recall the notion of a flasque bornological coarse space \cite[Def.~3.21]{buen}.
\begin{ddd}
$E$ vanishes on flasques if $E(X)\simeq 0$ for every flasque bornological coarse space $X$. 
\end{ddd}

In the following let $\cC$ denote the coarse structure of a bornological coarse space $X$.
For an entourage $U$ in $\cC$ we let $X_{U}$ be the bornological coarse space obtained from $X$ by replacing the coarse structure $\cC$ by the coarse structure $\cC\langle U\rangle$ generated by $U$. The identity map of the underlying set of $X$ is a morphism $X_{U}\to X$ of bornological coarse spaces. We get a natural morphism $$E(X)\to \lim_{U\in \cC}E(X_{U})\,.$$

\begin{ddd}
$E$ is $u$-continuous if for every bornological coarse space $X$ the natural morphism  
$$E(X)\to \lim_{U\in \cC} E(X_{U})$$ is an  equivalence.
\end{ddd}

Let $\bC$ be a complete stable $\infty$-category and consider a functor $E \colon \BC^{\op}\to \bC$.

\begin{ddd}\label{wefoijoi24r34tt}
$E$ is a coarse  cohomology theory if it has the following properties:
\begin{enumerate}
\item\label{fwiofjewofwfewfef3} $E$ is coarsely invariant.
\item\label{fwiofjewofwfewfef} $E$ is excisive.
\item $E$ vanishes on flasques.
\item\label{fwiofjewofwfewfef1} $E$ is $u$-continuous.
\end{enumerate}
\end{ddd}

\begin{rem}
In this remark we compare Definition \ref{wefoijoi24r34tt} with {Fukaya and Oguni's} definition of a coarse cohomology theory \cite[Def.~3.3]{Fukaya:2013aa}.

As usual in the {current}  coarse geometry literature they only consider proper metric spaces. In order to encode coarse invariance they introduce the coarse category: It is  obtained from the full subcategory of $\BC$ of proper metric spaces (where the coarse and bornological structures are induced from the metric) by identifying morphisms which are close to each other. Then a coarse cohomology theory {in the sense of Fukaya and Oguni is a} contravariant, $\Z$-graded, group-valued functor on the coarse category which vanishes on all spaces of the form $X \otimes {\nat_{can}}$ (where ${\nat_{can}}$ has the canonical metric structures) and satisfies a Mayer--Vietoris sequence for coarsely excisive decompositions.

If $E$ is an $\Sp$-valued  coarse  cohomology theory as in Definition \ref{wefoijoi24r34tt}, then we can derive a
coarse cohomology theory in the sense of {Fukaya--Oguni} by  taking homotopy groups and restricting to proper metric spaces. 
Condition \ref{wefoijoi24r34tt}.\ref{fwiofjewofwfewfef3} ensures that the resulting $\Z$-graded group-valued functor factorizes over the coarse category.
The excisiveness Condition  \ref{wefoijoi24r34tt}.\ref{fwiofjewofwfewfef} is stronger than
{satisfying a Mayer--Vietoris sequence for coarsely excisive decompositions, cf.~\cite[Lem.~3.38]{buen}.}
Finally, the Condition \ref{wefoijoi24r34tt}.\ref{fwiofjewofwfewfef1} is actually equivalent to the requirement that  $E(X \otimes \nat_{can})\simeq 0$, and $u$-continuity is not part of Fukaya and Oguni's axioms.
\end{rem}

If the $\infty$-category $\bC$ is complete and stable, then the opposite $\infty$-category $\bC^{\op}$ is cocomplete and stable.
If $E \colon \BC^{\op}\to \bC$ is a functor, then  we let $E^{\op} \colon \BC\to \bC^{\op}$ denote the induced functor between  the opposite categories.

Recall the definition of a   $\bC^{\op}$-valued coarse homology theory from \cite[Def.~4.22]{buen}. The following corollary immediately follows from a comparison of the definitions.

\begin{kor}\label{kor2435tre}
$E$ is a  $\bC$-valued coarse cohomology theory if and only if $E^{\op}$ is a $\bC^{\op}$-valued coarse homology theory.
\end{kor}

Using the correspondence between coarse homology theories and coarse cohomology theories we can transfer the results and definitions concerning coarse homology theories shown or stated in{, e.g.\ }\cite{buen} and \cite{ass} to the case of coarse cohomology theories. Here {are two examples} of such a transfer of definitions.

Recall the notion of a weakly flasque bornological coarse space \cite[Def.~4.18]{equicoarse}.
\begin{ddd}
A coarse cohomology theory $E$ is called strong if $E(X)\simeq 0$ for every weakly flasque bornological coarse space $X$.
\end{ddd}
Thus a coarse cohomology theory $E$    is strong if and only if $E^{\op}$ is a strong coarse homology theory.

Recall from \cite[Def.~6.3]{buen} that a coarse homology theory $F$ is called strongly additive if for every family $(X_{i})_{i\in I}$ of bornological coarse spaces we have an equivalence
$$F\big(\bigsqcup_{i\in I}^\free X_{i}\big) \simeq \prod_{i\in I} F(X_{i})$$
(induced by the collection of projection maps which exist by excision). Hence for coarse cohomology theories we get the following definition of {(strong)} additivity:

\begin{ddd}\label{defn_additivity}
A coarse cohomology theory $E$ is called strongly additive if for every   family $(X_{i})_{i\in I}$ of bornological coarse spaces we have an equivalence
\[\bigoplus_{i\in I} E(X_{i}) \simeq E\big(\bigsqcup^\free_{i\in I} X_{i}\big)\]
induced by the natural map.

We say that $E$ is additive, if the equivalence above is satisfied for all families of one-point spaces.
\end{ddd}


In \cite[Def.~4.3]{buen} we have introduced the stable $\infty$-category of coarse motivic spectra $\Sp\cX$ and the universal classical coarse homology theory $$\Yo^{s} \colon \BC\to \Sp\cX\, .$$
{For a complete stable $\infty$-category $\bC$ precomposition with $\Yo^{s}$ induces by \cite[Cor.~4.6]{buen}} an equivalence between the $\infty$-categories of colimit preserving functors from $\Sp\cX$ to $\bC^{\op}$ and   $\bC^{\op}$-valued coarse homology theories. By Corollary~\ref{kor2435tre} we have the analogous statement for coarse cohomology theories, as follows.

Let $\bC$ be a complete stable $\infty$-category.

\begin{kor}\label{kor2435zrterwe}
Precomposition by $\Yo^{s,\op}$ induces an equivalence between the $\infty$-categories of limit-preserving functors 
$\Sp\cX^{\op}\to \bC$ and  $\bC$-valued coarse cohomology theories.
\end{kor}

\subsection{Coarse cohomology theories by duality}
\label{sec124354trew}
 
In this section we provide a simple  construction of coarse cohomology theories by dualizing coarse homology theories.

 Let $\bC$ be a stable $\infty$-category and $C$ be an object of $\bC$. 
We assume one of the following:
\begin{enumerate}
\item\label{viowervwervwve1} 
$\bC$ is cocomplete and
   $F \colon \BC\to  \bC$ is a  coarse homology theory.  Then we use the notation  $$C^{(-)} \coloneqq \map(-,C) \colon \bC^{op}\to \Sp$$ for the mapping spectrum functor and set $\bD \coloneqq \Sp$.

 \item\label{viowervwervwveeee1}   $F \colon \BC\to \Sp$ is a   coarse homology theory.     In this case we assume that $\bC$ is complete and powered over $\Sp$. We write $$C^{(-)} \colon \Sp^{\op}\to \bC\, , \quad  A\mapsto C^{A}$$  
for the  power functor and set $\bD \coloneqq \bC$.
\item\label{rgioerp3g3g3} $\bC$ is complete and cocomplete and $F \colon \BC\to \bC$ is a  coarse homology theory. Furthermore,  $\bC$ is closed symmetric monoidal  and in particular admits a limit-preserving  internal
mapping object functor $$\underline{\map}(-,C) \colon \bC^{op}\to \bC\, .$$ In this case we write $C^{(-)} \coloneqq \underline{\map}(-,C)$ and set $\bD \coloneqq \bC$.
\end{enumerate}
 
 

\begin{ddd}\label{rgoirjgoiergergreg}
We define the functor
$$D_{C}(F) \colon \BC^{\op}\to \bD\, , \quad D_{C}(F) \coloneqq C^{(-)}\circ F^{\op}\, .$$
\end{ddd}

We consider $D_{C}(F)$ as the dual of $F$ with respect to $C$.
 
\begin{theorem}\label{wefoewpfewfewfwef}
$D_{C}(F)$ is a  $\bD$-valued coarse {cohomology} theory.
\end{theorem}

\begin{proof}
The functor $ C^{(-),\op}$ is colimit-preserving. The composition
$$ D_{C}(F)^{\op}=C^{(-),\op}\circ {F} \colon \BC \to \bD^{\op}$$ is therefore a 
  $\bD^{\op}$-valued coarse homology theory. 
Hence
$D_{C}(F)$ is a  $\bD $-valued coarse cohomology theory by Corollary \ref{kor2435tre}. 
\end{proof}


The proof of the following lemma is straightforward:
\begin{lem}\label{lem2435t5rgfder}
If $F$ is strong, then so is $D_{C}(F)$.
\end{lem}

In general $D_{C}(F)$ is not additive even if $F$ is additive, {as the next example shows.}

\begin{ex}
We consider the additive, $\Ch_{\infty}$-valued coarse ordinary homology theory $H\Q\cX$ with rational coefficients 
 (see \cite[Def.\,6.18]{buen} for the definition and \cite[Cor.\,6.24]{buen} for additivity) and the  object $\Q[0]$ in $\Ch_{\infty}$.
 Then, for every infinite set $I$,
\begin{eqnarray*}
D_{\Q[0]}(H\Q\cX )\big(\bigsqcup_I^\free *\big) &\simeq& \map\big(H\Q\cX^{\op}\big(\bigsqcup_I^\free *\big) ,\Q[0]\big)\\& \simeq &\map\big(\prod_I \Q[0] ,\Q[0]\big)\\&\stackrel{ \not\simeq}{\leftarrow}& \bigoplus_I  \map(\Q[0],\Q[0])\\&\simeq &\bigoplus_{I}D_{\Q[0]}(H\Q\cX)(*)
\end{eqnarray*}
showing that $D_{\Q[0]}(H\Q\cX )$ is not additive. 
Here $\map$ is the internal  mapping  object of $\Ch_{\infty}$.
\end{ex}

 We let $E \colon \BC^{\op}\to \bD$ be a coarse cohomology theory.

\begin{ddd}\label{rgiooerg34t34t}
A pairing between $E$ and $F$ with values in $C$ is a morphism of coarse cohomology theories
$p \colon E\to D_{C}(F)$.
\end{ddd}

In the examples above  
the identity provides a $C$-valued pairing between $D_{C}(F)$ and $F$. 
The examples    $\HAX$   from Theorem   \ref{iergerg} and
$Q_{C}$ from Theorem \ref{regior34t3t34t} further below come with a natural pairing with a corresponding homology theory.

\subsection{Coassembly maps}\label{wgowriepgwegweere}

In this section we translate some of the result from \cite{ass} from coarse homology to coarse cohomology theories.
A coarse homology theory  gives rise to a local homology theory   on the category of uniform bornological coarse spaces 
by pulling back along the cone functor. On the other hand, a local homology theory   can be coarsified to 
a coarse homology theory   by  composing with the Rips complex construction.  Applying these two constructions
in  a row  we get a new coarse homology theory which is connected with the original  one 
by a coarse assembly map. These constructions can be analogously performed with coarse cohomology theories and yield coarse coassembly maps. In the following we describe the details.

In \cite[Sec.\,2]{ass} we introduced the category of uniform bornological coarse spaces $\UBC$. The  notion of a local homology theory  is defined in \cite[Def.\,3.11]{ass}. The universal local homology theory  $$\Yo^{s}\cB \colon \UBC\to \Sp\cB$$  is constructed in  \cite[Sec.\,4]{ass}.
 Furthermore, in \cite[Sec.\,4.4]{equicoarse} we   introduced the universal strong coarse homology theory
$$\Yo^{s}_{\wfl} \colon \BC\to \Sp\cX_{\wfl}\, .$$ 
We let $$\cO^{\infty} \colon \UBC\to \Sp\cX$$ denote the germs-at-$\infty$ of the cone functor (\cite[Sec.~8]{ass} and \cite[Sec.~9.6]{equicoarse}).
By \cite[Lem.~9.5]{ass}   the composition $$\cO^{\infty}_{\wfl}  \colon \UBC\xrightarrow{\cO^{\infty}} \Sp\cX\to \Sp\cX_{\wfl}$$ is a local homology theory and therefore extends essentially uniquely along $\Yo^{s}\cB$ to a colimit-preserving functor (denoted by the same symbol) $$\cO^{\infty}_{\wfl} \colon \Sp\cB\to \Sp\cX_{\wfl}\, .$$

By \cite[Prop.~5.2]{ass} the Rips complex construction yields a   coarse homology theory $$\bP \colon \BC\to  \Sp\cB\, .$$ Therefore the composition
$$ \BC\xrightarrow{\bP}    \Sp\cB\xrightarrow{\cO^{\infty}_{\wfl}}  \Sp\cX_{\wfl}$$ is a  coarse homology theory.
We assume now that $E \colon \BC^{\op}\to \bC$ is a strong coarse cohomology theory.  Then we can interpret $E$ as a limit-preserving functor  defined on $\Sp\cX^{\op}_{\wfl}$. The composition $$E\cO^{\infty}\bP \coloneqq F\circ \cO^{\infty,\op}_{\wfl}\circ \bP^{\op} \colon \BC^{\op}\to \bC$$ is a new  coarse cohomology theory with the same target as $E$.
 It is related with $E$ by the coarse coassembly map, a natural transformation of functors
$$\mu^{E} \colon \Sigma^{-1} E \to E\cO^{\infty}\bP \, .$$
The construction of  the  coarse  coassembly map  is dual to the construction of the coarse  assembly map in the homological case
 \cite[Def.~9.7]{ass}. In detail $\mu^{E}$ is given  by composing $E$ with the transformation
$$\cO^{\infty}_{\wfl}\bP\xrightarrow{ \partial^{\Cone}\circ \bP} \Sigma \bF  \bP \xleftarrow{\simeq}  \Sigma    \, ,$$
of functors    $\Sp\cX \to \Sp\cX_{\wfl}$. Here   $\partial^{\Cone} \colon \cO^{\infty}_{\wfl}\to \Sigma \bF $ is   the  cone boundary 
 (see \cite[Def.~(9.1)]{ass}),     $\bF \colon \Sp\cB\to \Sp \cX_{\wfl}$ (see    \cite[Sec. 6]{ass}) is induced by forgetting the uniform structure, and we used    \cite[Prop.~6.2]{ass} for the second equivalence.  

%
%
%

In \cite[Thm's.\,1.3, 1.4 \& 1.5 ]{ass} we discussed  conditions implying that the coarse assembly map is an equivalence. Using the relation between coarse cohomology theories and coarse homology theories  these results yield conditions on  the bornological coarse space $X$ and   the coarse cohomology theory $E$ which imply that the coarse coassembly map
$$\mu^{E,X} \colon \Sigma^{-1}E(X)\to E\cO^{\infty}\bP(X)$$ is an equivalence.
We will only spell out the cohomological version of  \cite[Thm.\,1.3]{ass} and leave the translation of the other two theorems to the interested reader.

\begin{theorem}
 If $X$ admits a cofinal set of entourages $U$ such that $X_{U}$ has finite asymptotic dimension, then the coarse coassembly map $\mu^{E,X}$ is an equivalence.
\end{theorem}

\section{Coarse ordinary cohomology}
\label{sec243544345}

\subsection{Construction of \texorpdfstring{$\HAX$}{HAX}} 
In this section we construct coarse cohomology   $$\HAX \colon \BC\to \Ch_{\infty} $$ with coefficients in an abelian group $A$.
Its target is the presentable stable $\infty$-category of chain complexes defined as a Dwyer--Kan localization \begin{equation}
\label{eqrerz5645654}
 \iota \colon \Ch\to \Ch_{\infty}
 \end{equation}
 of the 
category $\Ch$ of chain complexes of abelian groups  at 
quasi-isomorphisms. 

To a set $X$ we can {functorially}   associate a simplicial set    $\hat X$,    the \v{C}ech nerve of the projection $X\to *$. For $n$ in $\nat$ its set of $n$-simplices   is  given by  $\hat X[n] \coloneqq X^{\times (n+1)}$.  

  Let $X$ be a bornological coarse space, $U$ be a coarse entourage of $X$, and $B$ be a bounded subset of $X$.
An $n$-simplex   $(x_{0},\dots,x_{n})$ in $\hat X$ is called $U$-controlled if 
 $(x_{i},x_{j})\in U$ for all $i,j$ in $[n]$. We say that this simplex is contained in $B$ if $x_{i}\in B$ for all $i$ in $[n]$.
 
 If {the entourage} $U$ contains the diagonal, then the $U$-{controlled} simplices form a simplicial subset $\hat X_{U}$ of $\hat X$. For simplicity,  we assume that all entourages appearing below contain the diagonal.
   
To any simplicial set $S$ and abelian group $A$ we can functorially associate a chain complex 
$C(S;A)$ in $\Ch$. It is defined as the chain complex associated to the cosimplicial abelian group $A^{S}$.
For $n$ in $\Z$ the group of $n$-chains is given by $C^{n}(S;A) \coloneqq A^{S[n]}$, and the boundary operator
$d \colon C^{n}(S;A)\to C^{n+1}(S;A)$ is given by $\sum_{i=0}^{n+1}(-1)^{i}d_{i}$, where
$d_{i}$ is induced by the $i$th face map $\partial_{i} \colon S[n+1]\to S[n]$. For example,
$\partial_{0}(x_{0},\dots,x_{n+1}) \coloneqq (x_{1},\dots,x_{n+1})$.

We   let $C_{U}(X,B;A)$ be the $\Z$-graded subgroup of $C(\hat X_{U};A)$ of functions which vanish on all simplices which are not contained in $B$. Observe that the $\Z$-graded subgroup $C_{U}(X,B;A)$ is not a subcomplex. Indeed, the differential of $C(\hat X_{U};A)$ restricts to maps
\begin{equation}\label{eqgrz56454643}
d \colon C_{U}(X,B;A)\to C_{U}(X,U[B];A)
\end{equation}
for all $B$ in the bornology $\cB$ of $X$. Hence, if we form the union of these $\Z$-graded subgroups over the bounded subsets $B$  of $X$, then we obtain {the following} subcomplex of $C(\hat X_{U};A)$:
\begin{equation}\label{eqrtrg435reds}
C_{U}(X;A) \coloneqq \colim_{B\in \cB} C_{U}(X,B;A)\, .
\end{equation}
Note that we consider $C_{U}(X;A)$ as an object of $\Ch$.

Let $f \colon X^{\prime}\to X$ be a morphism between bornological coarse spaces. The map $f$ induces a map of simplicial sets $\hat f \colon \hat X^{\prime}\to \hat X$.  Assume that
 $U^{\prime}$ is  an   entourage of $X^\prime$, $U$ is an entourage of $X$, and that $f(U^{\prime})\subseteq U$.
 Then $\hat f$  restricts to a map of simplicial sets
$\hat X^{\prime}_{U^{\prime}}\to \hat X_{U}$. Since $f$ is proper, pull-back along this map induces a morphism of chain complexes
$$f^{*} \colon C_{U}(X;A)\to C_{U^{\prime}}(X^{\prime};A)\,.$$
 
 We now want to perform the derived limit  of $C_{U}(X;A)$ over the entourages $U$ in the coarse structure  $\cC$ of $X$ {(see the Remark \ref{rem_HAX_derived_limit} for the reason why we want the limit to be derived)}.
 In order to produce a functor on the level of $\infty$-categories, technically we will work with right Kan extensions. 
 To this end 
 we consider the following category  $\BC^{\cC}$:
\begin{enumerate}
\item The objects of $\BC^{\cC}$   are pairs $(X,U)$ of a bornological coarse space  $X$ and a coarse entourage $U$ on $X$.
\item The morphisms $(X^{\prime},U^{\prime})\to (X,U)$ in $\BC^{\cC}$ are morphisms of bornological coarse spaces $f \colon X^{\prime}\to X$ such that $f(U^{\prime})\subseteq U$.
\end{enumerate}
We have functors
\begin{equation}\label{kfhiuiueiuwhfiuwefewf}
p \colon \BC^{\cC}\to \BC\,, \quad (X,U)\mapsto X
\end{equation}
and, using    $\iota$ from \eqref{eqrerz5645654},
$$\iota C(A) \colon (\BC^{\cC})^{\op}\to \Ch_{\infty}\, , \quad (X,U)\mapsto \iota C_{U}(X;A)\, .$$

\begin{ddd}\label{ewkfhweiofwefewf}
We define  the functor $\HAX \colon \BC^{\op}\to \Ch_{\infty}$  as the right Kan extension
\[\xymatrix{
(\BC^{\cC})^{\op}\ar[d]_-{{p^{\op}}}\ar[rr]^-{\iota C(A)} & & \Ch_{\infty}\\
\BC^{\op}\ar@{..>}[urr]_-{\ \HAX} & &
}\] of $\iota C(A)$ along $p^{\op}$.
\end{ddd}
  
\begin{rem}
  The point-wise formula for the right Kan extension provides the formula
  \begin{equation}\label{fijewoifjowefwef}
\HAX(X)\simeq \lim_{U\in \cC} \iota  C_{U}(X;A)
  \end{equation}
  for the evaluation of the functor $\HAX$ on the bornological coarse space $X$.   
  It is crucial to apply $\iota $ before taking the limit. This ensures that the limit is derived.
\end{rem}

\begin{theorem}\label{iergerg} The functor 
$\HAX \colon \BC^{\op}\to \Ch_{\infty}$ is a   coarse cohomology theory.
\end{theorem}

\begin{proof}
{The axioms given in Definition~\ref{wefoijoi24r34tt} for a coarse cohomology theory will be verified in the following four Lemmas \ref{lem24356rtewe}, \ref{oirejgoigergergerg}, \ref{efiwejfowfjewfwf} and \ref{rgioug0oergergeg}.}
\end{proof}

\begin{lem}\label{lem24356rtewe}
$\HAX$ is $u$-continuous
\end{lem}

\begin{proof}
Let $X$ be a bornological coarse space with coarse structure $\cC$. Using \eqref{fijewoifjowefwef} and a cofinality consideration we get the chain of canonical equivalences
\[\HAX(X) \simeq \lim_{U\in \cC} \iota C_{U}(X;A)\simeq \lim_{U\in \cC} \lim_{V\in \cC\langle U\rangle} \iota  C_{V}(X;A)\simeq \lim_{U\in \cC}   \HAX(X_{U})\,.\qedhere\]
\end{proof}

\begin{lem}\label{oirejgoigergergerg}
$\HAX$ is coarsely invariant.
\end{lem}

\begin{proof}
Let $f,g \colon X\to X^{\prime}$ be two maps of bornological coarse spaces which are close to each other. If $U$ is an entourage of $X$, 
then we choose an entourage $U^{\prime}$ of $X^{\prime}$  so large that $f(U)\subseteq U^{\prime}$, $g(U)\subseteq U^{\prime}$ and $(f,g)(\diag(X))\subseteq U^{\prime}$. 
Then for    all $n$ in $\nat$ and $i$ in $[n]$  we  have    the maps
$h_{i} \colon \hat X_{U}[n]\to \hat X_{U^{\prime}}^{\prime}[n+1]$ given by
$$(x_{0},\dots,x_{n})\mapsto (f(x_{0}),\dots,f(x_{i}),g(x_{i}),\dots, g(x_{n}))\, .$$
Pull-back along $h_{i}$ induces 
 a map
$$h_{i}^{n,*} \colon C_{U^{\prime}}^{n+1}(X^{\prime};A)\to C_{U}^{n}(X;A)\, . $$
We form $h^{n} \coloneqq \sum_{i=0}^{n} (-1)^{i}h_{i}^{n,*}$ and the  map
$h \coloneqq \bigoplus_{n\in \Z} h^{n}$ of degree $-1$.
Then one checks directly that
$$d\circ h+h\circ d=g^{*}-f^{*} \colon C_{U^{\prime}}(X^{\prime} ;A)\to C_{U }(X ;A)\, .$$

Let $X$ be a bornological coarse space. We consider the maps
$$p \colon \{0,1\}\otimes X\to X\, , \quad i \colon X\to \{0,1\}\otimes X$$
given by the projection and the inclusion of the point $0$, respectively. Then $p\circ i=\id_{X}$ and $i\circ p$ is close to the identity of $\{0,1\}\otimes X$.
For an entourage $U$ of $X$ let $\tilde U \coloneqq \{0,1\}^{2}\times U$ be the corresponding entourage of $\{0,1\}\otimes X$.
The above construction then shows that $(i\circ p)^{*}$ is chain homotopic to the identity on $C(\{0,1\}\otimes X;A)$.
This implies that
$$p^{*} \colon C_{U}(X;A)\to C_{\tilde U}(\{0,1\}\otimes X;A)$$
is an equivalence for every entourage $U$ of $X$. We conclude that
$$\HAX(p) \colon \HAX(X)\to \HAX(\{0,1\}\otimes X)$$
is an equivalence.
\end{proof}

\begin{lem}\label{efiwejfowfjewfwf}
$\HAX$ is excisive.
\end{lem}

\begin{proof}
Let $X$ be a bornological coarse space $X$. For an entourage $U$ and a subset $Y$ of $X$ we write
$U_{Y} \coloneqq U\cap (Y\times Y)$. We have a surjective restriction
$$C_{U}(X;A)\to C_{U_{Y}}(Y;A)\, .$$
We denote its kernel by
$C_{U}(X,Y;A)$.

Let $(Z,\cY)$ be a complementary pair on $X$ with $\cY=(Y_{i})_{i\in I}$. For every
$i$ in $I$  we consider the map of exact sequences $$\xymatrix{
0\ar[r]&C_{U}(X,Y_{i};A)\ar[r]\ar[d]^{r_{i}}&C_{U}(X;A)\ar[d]\ar[r]&C_{U_{Y_{i}}}(Y_{i};A)\ar[r]\ar[d]&0\\
0\ar[r]&C_{U_{Z}}(Z,Z\cap Y_{i};A)\ar[r]&C_{U_{Z}}(Z;A)\ar[r]&C_{U_{Z\cap Y_{i}}}(Z\cap Y_{i};A)\ar[r]&0
}\ ,$$
 where the vertical morphisms are the corresponding restriction maps.

We claim that the morphism $(r_{i})_{i}$ is an isomorphism of pro-systems indexed by $I$.
Indeed, we can define a map of complexes 
$s_{i} \colon C_{U_{Z}}(Z,Z\cap Y_{j};A)\to C_{U}(X,Y_{i};A)$ by extension by zero as long as $j$ in $I$ satisfies
$U[Y_{i}]\subseteq Y_{j}$.  Since $\cY$ is a big family, for every $i$ in $I$ we can choose such an index $j$ in $I$. Then the resulting family $(s_{i})_{i\in I}$ is an inverse of $(r_{i})_{i\in I}$.

The localization $\iota$ sends short exact sequences of chain complexes to fibre sequences. We apply $\iota$ and $\lim_{i\in I}$ and $\lim_{U\in \cC}$ in order to get the morphism between fibre sequences
\[\mathclap{\xymatrix{
\lim_{U\in \cC}\lim_{i\in I}\iota C_{U}(X,Y_{i};A)\ar[r]\ar[d]^{\simeq } & \HAX(X)\ar[d]\ar[r] & \HAX(\cY) \ar[d] \\
\lim_{U\in \cC}\lim_{i}\iota C_{U_{Z}}(Z,Z\cap Y_{i};A)\ar[r] & \HAX(Z)\ar[r] & \HAX(Z\cap \cY)  }}\ .\]
In view of the  left vertical equivalence  the right  square is cartesian.
\end{proof}
 
\begin{rem}\label{rem_HAX_derived_limit}
It is important for the proof of Lemma \ref{efiwejfowfjewfwf} that we consider the limits after applying $\iota$. Limits in $\Ch_{\infty}$ preserve fibre sequences. In contrast, limits in $\Ch$ in general do not preserve short exact sequences.
\end{rem}

\begin{lem}\label{rgioug0oergergeg}
$\HAX$ vanishes on flasques.
\end{lem}

\begin{proof}
Let $X$ be a flasque bornological coarse space with flasqueness implemented by the morphism $f \colon X\to X$. We define a map of chain complexes $$S \colon C_{V}(X;A)\to C_{U}(X;A)$$ by
$$S(\phi) \coloneqq \sum_{n=0}^{\infty} f^{n,*}\phi\, ,$$
where $U$ is an entourage of $X$ and  $V \coloneqq \bigcup_{n\in \nat} f^{n}(U)$. 
Since $\phi$ is supported on some bounded subset of $X$ almost all summands vanish and the sum has a well-defined interpretation. One furthermore checks that
$$f^{*}\circ S +r=S\, ,$$
where $r \colon C_{V}(X;A)\to  C_{U}(X;A)$ is the restriction. Applying $\iota$ and $\lim_{U\in \cC}$, the morphisms $S$ for various $U$   induce a morphism of chain complexes
$$\tilde S \colon \HAX(X)\to \HAX(X)$$
satisfying
\begin{equation}\label{fwefklj234or23r23r}
\HAX(f)\circ \tilde S+\id_{\HAX(X)}\simeq \tilde S\, .
\end{equation} 
Since $\HAX$ is coarsely invariant by Lemma  \ref{oirejgoigergergerg} we {get} the equivalence
\begin{equation}\label{fwefklj234or23r23r1}
\tilde S+\id_{\HAX(X)}\simeq \tilde S\, .
\end{equation}
It implies $\HAX(X)\simeq 0$.
\end{proof}

We have thus verified all four axioms for coarse cohomology theories and hence finished the proof of Theorem \ref{iergerg}.

In the remaining part of this section we will verify that $\HAX$ is a strong coarse cohomology theory that satisfies additivity, {and in the following section} we compare our definition to the original one of Roe \cite{roe_coarse_cohomology} and describe a natural pairing with coarse homology.

\begin{lem}
$\HAX$ is strong.
\end{lem}

\begin{proof}
We assume that $X$ is weakly flasque {and we} repeat the argument for Lemma~\ref{rgioug0oergergeg}. Because we already know that $\HAX$ is a coarse cohomology theory, {in order} to see that~\eqref{fwefklj234or23r23r} implies \eqref{fwefklj234or23r23r1} it suffices to {have} $\Yo^{s}(f)\simeq \id_{\Yo^{s}(X)}$.
\end{proof}



Recall Definition \ref{defn_additivity} of (strong) additivity and  the definition     of a free union (see \cite[Def.\,2.27]{buen}). 
The following result 
implies additivity of $\HAX$, but because of the additional assumption it is weaker than strong additivity.

Let $(X_{i})_{i\in I}$ be a family of bornological coarse spaces.

\begin{lem}\label{qijrfoergwregrgrefwg}
If $X_{i}$ has a maximal coarse entourage for every $i$ in $I$, then  we have an equivalence $$ \bigoplus_{i\in I} \HAX(X_{i})\xrightarrow{\simeq} \HAX\big(\bigsqcup^{\free}_{i\in I}X_{i}\big)\, .$$
 \end{lem}
 
\begin{proof}
The set of entourages of $X$  of the form
$\bigsqcup_{i\in I}U_{i}$  for families $(U_{i})_{i\in I}$ in $ \prod_{i\in I}\cC_{i}$   (where $\cC_{i}$ denotes the coarse structure of $X_{i}$)
is cofinal in the set  $\cC$ of entourages of $X$. Let $\cB$ denote the bornology of $X$ and let $\cB_{i}$ denote the bornology of $X_{i}$ for every $i$ in $I$. For any $B$ in $\cB$ we have $B\cap X_{i}\in \cB_{i}$ for all $i$ in $I$,  and $B\cap X_{i}=\emptyset$ for all but finitely many $i$ in $I$.  
Consequently, we get the following chain of equivalences:
\begin{eqnarray*}
\HAX(X)&\simeq& \lim_{U\in \cC} \iota \colim_{B\in \cB}  C_{U}(X,B;A)\\&\simeq&
\lim_{(U_{i})_{i\in I}\in \prod_{i\in I}\cC_{i}} \iota \colim_{B\in \cB}  C_{(U_{i})_{i\in I}}(X,B;A)\\&\stackrel{!}{\simeq}& \lim_{(U_{i})_{i\in I}\in \prod_{i\in I}\cC_{i}} 
\iota \colim_{B\in \cB}  \bigoplus_{i\in I} C_{U_{i} }(X_{i},B\cap X_{i};A)\\&\simeq&\lim_{(U_{i})_{i\in I}\in \prod_{i\in I}\cC_{i}} 
\iota    \bigoplus_{i\in I} \colim_{B_{i}\in \cB_{i}}C_{ U_{i} }(X_{i},B_{i} ;A)\\&\simeq&
\lim_{(U_{i})_{i\in I}\in \prod_{i\in I}\cC_{i}} 
\iota    \bigoplus_{i\in I} C_{ U_{i}  }(X_{i} ;A)\\
&\simeq&
\lim_{(U_{i})_{i\in I}\in \prod_{i\in I}\cC_{i}} 
    \bigoplus_{i\in I} \iota  C_{ U_{i} }(X_{i} ;A)\\
    &\stackrel{\simeq}{\leftarrow} &   \bigoplus_{i\in I} \lim_{U_{i}\in \cC_{i}}\iota  C_{U_{i}}(X_{i} ;A)\\
    &\simeq&
    \bigoplus_{i\in I}   \HAX(X_{i})
    \end{eqnarray*}
In addition  to  the properties of the bounded subsets of $X$ mentioned above, for the marked equivalence  we also use the fact that  the chain boundary operator does not mix the different coarse components of $X$. 
In order to see that the left-pointing arrow is an equivalence we use the additional assumption of the lemma to replace the limits by the evaluations at the corresponding maximal entourages.
\end{proof}
\begin{rem}\label{rem_HAX_strongly_add}
We have written the argument for the Lemma \ref{qijrfoergwregrgrefwg} in greater detail than necessary in order to locate the problem that appears  if one wants to show strong additivity, i.e., if one drops the additional assumption. Then in order to show that the left-pointing arrow is an equivalence one must 
  distribute a cofiltered limit over an infinite sum \cite[Rem.\,6.28]{buen}, see also the proof of Lemma \ref{wergijowergerfwffr} below. It is not clear that this is alway  possible in $\Ch_{\infty}$. 
\end{rem}

%
%
%


\subsection{Further properties}

If $X$ is a metric space, then Roe \cite{roe_coarse_cohomology} has defined coarse cohomology groups $\HX_{Roe}^{*}(X;A)$. Roe's coarse cohomology groups are defined as the cohomology groups of the complex $\CX_{Roe}(X;A)$ of locally bounded, $A$-valued Borel functions on the simplicial space
$\hat X$ whose restrictions to $\hat X_{U_{r}}$ have bounded support for every entourage $U_{r} \coloneqq \{x,y \in X \ | \ d(x,y) < r\}$. In our notation  
$$\CX_{Roe}(X;A) \coloneqq \lim_{U\in \cC} C_{U,Roe}(X;A)\, ,$$
where $C_{U,Roe}(X;A)$ is the subcomplex of $C_{U}(X;A)$ of the locally bounded Borel functions. We have a natural morphism
\begin{equation}\label{wkljvwovwevwevwedwedewdewd}
\iota \CX_{Roe}(X;A)\to \HAX(X)\, .
\end{equation}
 
\begin{lem}\label{lem23435tzrgewrt}
If $X$ is a proper metric space, then the morphism \eqref{wkljvwovwevwevwedwedewdewd} is an equivalence.   
\end{lem}

\begin{proof}
It suffices to show that \eqref{wkljvwovwevwevwedwedewdewd} induces a quasi-isomorphism. 
Since the domain and the target of this morphism are coarsely invariant functors we can replace $X$ by a locally finite, discrete subset which is coarsely equivalent to $X$.  Furthermore, we can replace the limit over $U $ in $\cC$ by the limit over the family of entourages $U_{n} \coloneqq \{x,y \in X \ | \ d(x,y) < n\}$ indexed by $n\in \nat$.

If $X$ is a locally finite, discrete metric space, then the conditions of being a Borel function  and of being locally bounded are vacuous. In this case the only difference between Roe's complex and our complex is the order of the limit $\lim_{U\in \cC}$ and the localization $\iota$. In Roe's case the limit is not derived.

We now observe that the restriction maps  $C_{U_{n+1}}(X;A)\to C_{U_{n}}(X;A)$ are surjective for all $n$ in $\nat$. 
As explained below this   implies that one can interchange the order of taking the limit and the localization so that the  assertion  of the lemma follows.

In order to see that we can   interchange the order of taking the limit and the localization we model $\Ch_{\infty}$ using the model category structure on $\Ch $ whose weak equivalences are quasi-isomorphisms and whose fibrations are degree-wise
epimorphisms. On $\Fun(\nat^{\op},\Ch)$ we then have a Reedy model category structure in which the diagram  $(C_{U_{n}}(X;A))_{n\in \nat^{\op}}$ is fibrant.  This implies that  the canonical morphism
$$ \iota  \lim_{\nat^{\op}} C_{U_{n}}(X;A)  \xrightarrow{\simeq} \lim_{\nat^{\op}}   \iota C_{U_{n}}(X;A)$$
is an equivalence.
\end{proof}

Therefore our construction extends Roe's coarse cohomology from proper metric spaces to all bornological coarse spaces.

  We now describe a natural pairing with values in $\iota A[0]$ in the sense of Definition \ref{rgiooerg34t34t} between the coarse cohomology $\HAX$ and the coarse homology theory 
with $\Z$-coefficients $\HX^{\hlg}$ from \cite[Def.\ 6.18]{buen}\footnote{In the reference $\HX^{\hlg}$ is denoted by $\HX$.}. 
 
We use the closed symmetric monoidal structure of the $\infty$-category $\Ch_{\infty}$.
The dualizing object (denoted by $C$ in Section \ref{sec124354trew})  is the object
$\iota A[0]$   in $\Ch_{\infty}$, where $A[0]$ in $\Ch$ is the chain complex with $A$ in degree zero. 

 Recall that the coarse homology     of the bornological coarse space $X$ is given by
$$\HX^{\hlg}(X)\simeq \iota \CX^{\hlg}(X)\, ,  $$
where $\CX^{\hlg}(X)$ is the complex of locally finite and controlled chains \cite[Def.~6.17]{buen}.
Let $U$ be a coarse entourage {of $X$.} Then we let
$$\widetilde{\CX}^{\hlg}(X,U) \coloneqq \CX_{U}^{\hlg}(X)$$ be the subcomplex  of $\CX^{\hlg}(X)$ of the locally finite, $U$-controlled chains.
We thus get a functor 
$$\widetilde{\CX}^{\hlg} \colon \BC^{\cC}\to  \Ch\, , \quad (X,U)\mapsto\widetilde{\CX}^{\hlg}(X,U)\, .$$
Furthermore, recall the functor
$$C(A) \colon {(\BC^{\cC})^{\op}}\to  \Ch\, , \quad (X,U)\mapsto C(A)(X,U) \coloneqq C_{U}(X;A)\, .$$
 We first define a  natural transformation of $\Ch$-valued functors {${(\BC^{\cC})^{\op}}\to  \Ch$}
$$\tilde p \colon C(A)\to \underline{\Hom}( {\widetilde{\CX}^{\hlg,\op}},A[0])$$  as follows.
Let $(X,U)$ be an object of $\BC^{\cC}$, fix an $n$ in $\nat$, and let $\phi$ be an element of $C_{A}(X,U)^{n}$.
Then we define the homomorphism $\tilde p_{(X,U)}(\phi)$ in $\Hom(\widetilde{ \CX}^{\hlg,\op}(X,U),A[0])^{n}$
 as the  $\Z$-linear extension of the map which  sends  the simplex $( x_{0},\dots,x_{n})  $ in $\hat X_{U}[n]$ to $\phi((x_{0},\dots,x_{n}))$ and vanishes on simplices of dimensions different from $n$.
 One easily checks that $\tilde p_{(X,U)}$ is a map of chain complexes, and that
 the collection of maps $\tilde p_{(X,U)}$ for all $(X,U)$ in $\BC^{\cC}$ defines a natural transformation of functors $\tilde p$.

The natural transformation $\tilde p$ induces a morphism
$$\iota \tilde  p \colon \iota C(A)\to \iota\underline{\Hom}( {\widetilde{\CX}^{\hlg,\op}},A[0])$$
between functors from $ {(\BC^{\cC})^{\op}}$ to $\Ch_{\infty}$. We derive the  desired pairing
\begin{equation}\label{wqdoi12de}
p \colon \HAX \to D_{\iota A[0]}(\HX^{\hlg})
\end{equation} 
by a right Kan extension of $\iota \tilde  p$
along the forgetful functor \eqref{kfhiuiueiuwhfiuwefewf} from $ {(\BC^{\cC})^{\op}}$ to $ {\BC^{\op}}$. To this end we must check that the domain and target of this extension are the correct functors. 

By Definition \ref{ewkfhweiofwefewf}
the domain of the Kan extension of $\iota \tilde p$ is $\HAX$. We now evaluate the  target on $X$ in $\BC$.
 Using the  objectwise formula for the Kan extension 
this evaluation is given by 
\begin{align*}
\mathclap{
\lim_{U\in \cC}  \iota\underline{\Hom}(\widetilde{\CX}^{\hlg}(X,U),A[0])\simeq  
\lim_{U\in \cC}   \underline{\map}(\iota \widetilde{\CX}^{\hlg}(X,U),\iota A[0])\simeq
  \underline{\map}(\colim_{U\in \cC}\iota \widetilde{\CX}^{\hlg}(X,U),\iota A[0])\, .
}\end{align*}
We now use 
 the chain of equivalences 
$$\colim_{U\in \cC}  \iota \widetilde{\CX}^{\hlg}(X,U) \simeq  \iota \colim_{U\in \cC} \widetilde{\CX}^{\hlg}(X,U)\simeq
\iota \CX^{\hlg}(X)\simeq \HX^{\hlg}(X)\, ,
$$ where the first equivalence follows from the fact that the poset   $\cC$ of entourages  of $X$ is filtered and $\iota$ commutes with filtered colimits.
We therefore obtain the following formula for the target:
$$ \underline{\map}(\HX^{\hlg}(X),\iota A[0])\simeq D_{\iota A[0]}(\HX^{\hlg})(X)\, .$$
The right Kan extension of $\iota \tilde p$ therefore is a morphism as in \eqref{wqdoi12de} and this is the desired pairing.

\section{The coarse cohomology theory \texorpdfstring{$\boldsymbol{Q_{C}}$}{QC}}
\label{secjk34tgfderfdsrf}

\subsection{The construction}
\label{f2h23iorj2r32r32r2}

In this section we introduce the $\bC$-valued coarse cohomology theory $Q_C$ for $C$ an object of  a  complete and cocomplete  stable $\infty$-category $\bC$ which is tensored and cotensored over $\Spc$. This coarse cohomology theory can be thought of as a generalized version of coarse stable cohomotopy {which is the special case $\bC=\Sp$ and the sphere spectrum $C \coloneqq S$}. We prove that $Q_C$ is a strong coarse cohomology theory and in Section \ref{sec_further_props_QC} we discuss a natural pairing with coarse stable homotopy which was introduced in \cite[Def.~6.29]{buen}.

If $X$ is a bornological coarse space and $U$ is a coarse entourage of $X$, then $P_{U}(X)$ denotes the space of probability measures on the discrete measurable space $X$ {with} finite, $U$-bounded support. This space has the structure of a simplicial complex, and it is a (quasi-)metric space with the path (quasi-)metric induced by the spherical metric on the simplices.\footnote{For a quasi-metric we allow infinite distances.}
We actually have a functor
$$\BC^{\cC}\to \Top\, , \quad (X,U)\mapsto P_{U}(X)$$
(see Section \ref{sec243544345} for the definition of $\BC^{\cC}$).
Let 
\begin{equation}\label{wqwfwedqewdqwdq}
\iota \colon \Top\to \Spc
\end{equation}  
be the Dwyer--Kan localization of $\Top$ at the weak equivalences, where $\Spc$ denotes the $\infty$-category of spaces.

Assume that $\bC$ is a   complete and cocomplete $\infty$-category which is tensored and cotensored over $\Spc$. In particular, any  object $C$  of $\bC$ gives rise to a functor
\begin{equation}\label{gjfglkjglkjgergergereg}
C^{(-)} \colon \Spc^{\op}\to \bC\, , \quad A\mapsto C^{A}\, .
\end{equation}
We shall define a functor
$$Q_{C} \colon \BC^{\op}\to \bC$$
whose evaluation on objects is given by
\begin{equation}\label{ewfweojo2p4jrpwfewrfw23fffwefwefwef}
X\mapsto \lim_{U\in \cC} \colim_{B\in \cB} \Fib(C^{\iota P_{U}(X)}\to C^{\iota P_{U}(X\setminus B)})\, .
\end{equation}
 
To this end we consider the category $\BC^{\cC,\cB}$:
\begin{enumerate}
\item An object of $\BC^{\cC,\cB}$ is a triple $(X,U,B)$ of a bornological coarse space $X$, a coarse entourage $U$ of $X$, and a bounded subset $B$ of $X$. 
\item A morphism $f \colon (X^{\prime},U^{\prime},B^{\prime})\to (X,U,B)$ is a morphism of bornological coarse spaces $f \colon X^{\prime}\to X$ such that ${f }(U^{\prime})\subseteq U$ and {$f^{-1}(B) \subseteq B^{\prime}$.}
\end{enumerate}
We have forgetful functors
$$\BC^{\cC,\cB}\stackrel{q}{\to} \BC^{\cC}\stackrel{p}{\to} \BC\, , \quad (X,U,B)\mapsto (X,U)\mapsto X\, .$$
We furthermore have a functor 
$$W \colon \BC^{\cC,\cB}\to \Top^{\Delta^{1}}\, , \quad W(X,U,B) \coloneqq (P_{U}(X\setminus B)\to P_{U}(X))\, .$$
It induces the functor 
$$\tilde Q_{C} \colon (\BC^{\cC,\cB})^{\op}\to \bC\, , \quad \tilde Q_{C}(X,U,B) \coloneqq \Fib(C^{\iota W}) \, .$$
\begin{ddd}\label{wrowpt32432424234234}
We define the functor $Q_{C}$ as the composition of a left and of a right Kan extension {as follows:}
$$\xymatrix{(\BC^{\cC,\cB})^{\op}\ar[d]_-{q}\ar[rr]^-{\tilde Q_{C}}&&\bC
\\(\BC^{\cC})^{\op}\ar[d]_-{p}\ar@{..>}[urr]^{\hat Q_{C}}&\\
\BC^{\op}\ar@{-->}[uurr]_-{Q_{C}}}$$
\end{ddd}

\begin{theorem}\label{regior34t3t34t}
$Q_{C}$ is a {$\bC$-valued} coarse cohomology  theory.
\end{theorem}

\begin{proof}
{In the following four Lemmas \ref{lemrgfdsfdv}, \ref{lem564343553}, \ref{lem76543215423} and \ref{goerpgergrege} we verify the four axioms from Definition \ref{wefoijoi24r34tt} on coarse cohomology theories.}
\end{proof}

\begin{lem}\label{lemrgfdsfdv}
$Q_{C}$ is $u$-continuous.
\end{lem}

\begin{proof}
We have
$\hat Q_{C}(X,U)\simeq \colim_{B\in \cB} \tilde Q_{C}(X,U,B)$. Then, {by \eqref{ewfweojo2p4jrpwfewrfw23fffwefwefwef},}
\[Q_{C}(X)\simeq \lim_{U\in \cC} \hat Q_{C}({X,U}) \simeq  \lim_{U\in \cC} \lim_{n\in \nat}\hat  Q_{C}(X,U^{n})\simeq\lim_{U\in \cC}Q_{C}(X_{U})\, .\qedhere\]
\end{proof}

\begin{lem}\label{lem564343553}
$Q_{C}$ is coarsely invariant.
\end{lem}

\begin{proof}
For a coarse entourage $U$  of $X$ we form the entourage $\tilde U \coloneqq \{0,1\}^{2}\times U$ of $\{0,1\}\times X$.
The projection
$P_{\tilde U}(\{0,1\}\otimes Y)\to P_{U}(Y)$ is a homotopy equivalence for every subset $Y$ of $X$.
For every bounded subset $B$ of $X$ we define the bounded subset $\tilde B \coloneqq \{0,1\}\times B$ of $\{0,1\}\otimes X$.
 Then
$\tilde Q_{C}(X,U,B)\to \tilde Q_{C}(\{0,1\}\otimes X,\tilde U,\tilde B)$ is an equivalence for every $B$ in $\cB$ and 
$U$ in $\cC$. We get an equivalence after applying $\lim_{U\in \cC}\colim_{B\in \cB}$. Since the bounded subsets of the form $\tilde B$ for $B$ in $\cB$ and the entourages of the form $\tilde U$ for $U$ in $\cC$ are cofinal in the bounded subsets or entourages, respectively, of $\{0,1\}\otimes X$ we get the desired equivalence
$Q_{C}(X)\to Q_{C}(\{0,1\}\otimes X)$.  
\end{proof}

\begin{lem}\label{lem76543215423}
$Q_{C}$ is excisive.
\end{lem}

\begin{proof}
Let $\cY \coloneqq (Y_{i})_{i\in I}$ be a big family on $X$ and let $(\cY,Z)$ be a complementary pair. Let $W$ be a subset of $X$. If $i$ is sufficiently large, then $(Y_{i},Z)$ is a $U$-covering of $X$, i.e., every $U$-bounded subset of $X$ is contained in at least one of $Y_{i}$ or $Z$. In this case
$$\xymatrix{P_{U}(W\cap Z\cap Y_{i})\ar[r]\ar[d]&P_{U}(W\cap Z)\ar[d]\\P_{U}(W\cap Y_{i})\ar[r]&P_{U}(W)}$$
is a homotopy cocartesian diagram since it is cocartesian and all maps are inclusions of subcomplexes.
It follows that
$$\xymatrix{C^{\iota P_{U}( W)}\ar[r]\ar[d]&C^{\iota P_{U}(W\cap Z)}\ar[d]\\C^{\iota P_{U}(W\cap Y_{i})}\ar[r]&C^{\iota P_{U}(W\cap Z\cap Y_{i})}}$$
is cartesian, from which we conclude that
$$\xymatrix{\tilde Q_{C}(X,U,B)\ar[r]\ar[d]&\tilde Q_{C}(Z,U,B)\ar[d]\\\tilde  Q_{C}(Y_{i},U,B)\ar[r]&\tilde  Q_{C}(Z\cap Y_{i} ,U,B)}$$
is cartesian.
We apply $\lim_{i\in I}\lim_{U\in \cC} \colim_{B\in \cB}$ and get a square $$\xymatrix{Q_{C}(X)\ar[r]\ar[d]&Q_{C}(Z)\ar[d]\\Q_{C}(\cY)\ar[r]&Q_{C}(Z\cap \cY)}$$
in $\bC$. We can interchange the order of taking the limits, i.e., apply $\lim_{U\in \cC} \lim_{i\in I} \colim_{B\in \cB}$ without changing the result. For every $U$ in $\cC$ let $I(U)$ be the subset of those $i$ in $I$ such that $(Y_{i},Z)$ is a $U$-covering.
By cofinality, we can restrict the limit to $\lim_{U\in \cC} \lim_{i\in I(U)} \colim_{B\in \cB}$.
Then the square above is obtained by applying this operation to a diagram of cartesian squares and,  by  stability of $\bC$ in order to deal with the colimit, is itself cartesian.
\end{proof}

\begin{lem}\label{goerpgergrege}
$Q_{C}$ vanishes on flasques.
\end{lem}

\begin{proof}
Let $X$ be a flasque bornological coarse space with flasqueness implemented by the morphism $f \colon X\to X$. We write
$$F_{U}(B) \coloneqq \tilde Q_{C}(X,U,B)\, . $$ 
Note that by definition 
$$Q_{C}(X)\simeq \lim_{U\in \cC} \colim_{B\in \cB} F_{U}(B)\, . $$
For an entourage $U$ of $X$ we define
$\tilde U \coloneqq \bigcup_{n\in \nat} f^{n}(U)$.
We then have the diagram
$$\xymatrix{
\colim_{B\in \cB, B \cap f^{1}(X)=\emptyset} F_{\tilde U}(B)\ar[d]\ar[rrr]^-{f^{0,*}}&&& \colim_{B\in \cB} F_{ U}( B)\\
\colim_{B\in \cB, B\cap f^{2}(X)=\emptyset} F_{\tilde U}(B)\ar[rrr]^-{f^{0,*}+f^{1,*}} \ar[d]&&& \colim_{B\in \cB} F_{  U}( B)\ar@{=}[u]\\
\colim_{B\in \cB, B \cap f^{3}(X)=\emptyset} F_{\tilde U}(B)\ar[rrr]^-{f^{0,*}+f^{1,*}+f^{2,*}} \ar[d]&&& \colim_{B\in \cB} F_{U}( B)\ar@{=}[u]\ar@{=}[d]\\
\vdots\ar[d]&&& \vdots \ar@{=}[d]\\
\colim_{B\in \cB} F_{\tilde U}(B)\ar[rrr]^-{s_{U}}&&& \colim_{B\in \cB} F_{U}( B)
}$$
 The squares commute since the composition
 $$ \colim_{B\in \cB, B \cap f^{n-1}(X)=\emptyset} F_{\tilde U}(B)\to  \colim_{B\in \cB, B \cap f^{n}(X)=\emptyset} F_{\tilde U}(B)\xrightarrow{f^{n-1,*}}  \colim_{B\in \cB} F_{ U}( B)$$ has a preferred equivalence to zero.
 The map $s_{U}$ is induced. If $U^{\prime}$ is a second entourage of $X$ such that $U\subseteq U^{\prime}$, then we have a natural commuting diagram
 $$\xymatrix{ \colim_{B\in \cB} F_{\tilde U^{\prime}}(B) \ar[r]^{s_{U^{\prime}}}\ar[d]& \colim_{B\in \cB} F_{  U^{\prime}}(B)\ar[d]\\ \colim_{B\in \cB} F_{\tilde U}(B) \ar[r]^{s_{U}}& \colim_{B\in \cB} F_{  U}(B)}$$
More precisely, one can perform the construction above in diagrams indexed by the poset~$\cC$. The construction then yields an interpretation of the family of morphisms $(s_{U})_{U\in \cC}$ as a morphism between diagrams. By applying $\lim_{U\in \cC}$ we get a morphism
$$s \colon Q_{C}(X)\to Q_{C}(X)\, .$$
By construction it satisfies
\begin{equation}\label{dqwdqwdqwdqwd}
Q_{C}(f)\circ s+\id_{Q_{C}(X)}\simeq s\, .
\end{equation} 
Since $Q_{C}$ is coarsely invariant we conclude that
\begin{equation}\label{dqwdqwdqwdqwd1}
 s+\id_{Q_{C}(X)}\simeq s
 \end{equation}  and therefore
$Q_{C}(X)\simeq 0$.
\end{proof}

\subsection{Further properties of \texorpdfstring{$Q_{C}$}{Q-C}}
\label{sec_further_props_QC}

In the next lemmas we establish that $Q_C$ is strong and strongly additive.  We furthermore describe the natural pairing with coarse stable homotopy.

\begin{lem}
$Q_{C}$ is strong.
\end{lem}

\begin{proof}
Let $X$ be weakly flasque. We repeat the argument for Lemma \ref{goerpgergrege}. Since we already know that $Q_{C}$ is a coarse cohomology theory, in order to see that \eqref{dqwdqwdqwdqwd} implies \eqref{dqwdqwdqwdqwd1} we only need that
$\Yo^{s}(f)\simeq \id_{\Yo^{s}(X)}$.
\end{proof}

Recall the Definition \ref{defn_additivity} of strong additivity.  
\begin{lem}\label{wergijowergerfwffr}  If  $\bC$ has the property that  cofiltered limits distribute over coproducts\footnote{see e.g.  \cite[Rem.\,6.28]{buen}}, then
$Q_{C}$ is strongly additive.
\end{lem}

\begin{proof}
Let $(X_{i})_{i\in I}$ be a family of bornological coarse spaces and \begin{equation}\label{oijgoregerg}
U \coloneqq \bigsqcup_{i\in I}U_{i}
\end{equation}
be an entourage of the free union   \cite[Def.~2.27]{buen} 
$$X \coloneqq \bigsqcup_{i\in I}^{\free} X_{i}\, .$$
Then we have an isomorphism of topological spaces
$$P_{U}(X)\cong \coprod_{i\in I} P_{U_{i}}(X_{i})\, .$$ 
A subset $B$ of  $X$ is   bounded if and only if 
$B_{i} \coloneqq B\cap X_{i} $  is bounded for all $i$ in $I$ and empty for all but finitely many $i$ in $I$.
We conclude that
\[\mathclap{
\tilde Q_{C}(X,U,B)\simeq \Fib(C^{\iota P_{U}(X)}\to C^{\iota P_{U}(X\setminus B)})\simeq \bigoplus_{i\in I}  \Fib(C^{\iota P_{U_{i}}(X_{i})}\to C^{\iota P_{U_{i}}(X_{i}\setminus B_{i})})\simeq \bigoplus_{i\in I}  \tilde Q_{C}(X_{i},U_{i},B_{i})\, .
}\]
We get the equivalence
\begin{align*}
\mathclap{
\hat Q_{C}(X,U)\simeq \colim_{B\in \cB} \tilde Q_{C}(X,U,B) \simeq   \colim_{B\in \cB} \bigoplus_{i\in I} \tilde Q_{C}(X_{i},U_{i},B_{i})  \simeq \bigoplus_{i\in I} \colim_{B_{i}\in \cB_{i}} \tilde Q_{C}(X_{i},U_{i},B_{i})\simeq \bigoplus_{i\in I} \hat Q_{C}(X_{i},U_{i}) \, ,
}
\end{align*}
where $\cB$ is the poset of bounded subsets of $X$.  The subset of entourages of the form \eqref{oijgoregerg}
is cofinal in the coarse structure $\cC$ of $X$. In the definition of $Q_{C}(X)$  we can therefore restrict the limit over  $\cC$ to this  set and  get the equivalence
\begin{align*}
\mathclap{
Q_{C}(X)\simeq \lim_{U\in \cC} \hat Q_{C}(X,U) \simeq \lim_{(U_{i})_{i\in I}\in \prod_{i\in I}\cC_{i}} \bigoplus_{i\in I} \hat Q_{C}(X_{i},U_{i}) \stackrel{!}{\simeq }\bigoplus_{i\in I} \lim_{U_{i}\in \cC_{i}} \hat Q_{C}(X_{i},U_{i}) \simeq \bigoplus_{i\in I} Q_{C}(X_{i})\, .
}
\end{align*}
In the marked equivalence we use the assumption on $\bC$.
One   checks that this equivalence is indeed induced by the collection of morphisms $Q_{C}(X_{i})\to Q_{C}(X)$ for all $i$ in $I$ given by excision for the complementary pair $(X_{i},\{X\setminus X_{i}\})$ on $X$.
\end{proof}

We  now describe   a $C$-valued pairing $$p \colon Q_{C}\to D_{C}(Q^{\hlg})$$ in the sense of Definition \ref{rgiooerg34t34t},
where $Q^{\hlg}(X)$ is the coarse stable homotopy theory of $X$ \cite[Def. 6.29]{buen}.

We first recall the definition of $Q^{\hlg}$.
We start with the functor
$$\BC^{\cC,\cB}\to \Top^{\Delta^{1}}\, , \quad (X,U,B)\mapsto (P_{U}(X\setminus B)\to P_{U}(X))\, .$$ 
We apply the localization functor $\iota \colon \Top\to \Spc$, the stabilization functor  $\Sigma^{\infty}_{+} \colon \Spc\to \Sp$, and finally the cofibre functor in order to    get the functor 
$$\tilde Q^{\hlg} \colon \BC^{\cC,\cB}\to \Sp\, , \quad \tilde Q^{\hlg}(X,U,B)\simeq   \Cofib(\Sigma^{\infty}_{+} \iota P_{U}(X\setminus B)\to \Sigma^{\infty}_{+} \iota  P_{U}(X))\, .$$
Similarly as in Definiton \ref{wrowpt32432424234234}, the coarse homology theory $Q^{\hlg}$ is obtained
as the composition of a right and of a left Kan extension
\[\xymatrix{
\BC^{\cC,\cB}\ar[d]\ar[rr]^-{\tilde Q^{\hlg}}&&{\Sp}\\
\BC^{\cC} \ar[d]\ar@{..>}[urr]^{\hat Q^{\hlg}}&\\
\BC\ar@{-->}[uurr]_-{Q^{\hlg}}
}\]

Since $\bC$ is stable, the power structure  of $\bC$ over $\Spc$ extends to a power structure over $\Sp$.
If we fix  the object $C$ in $\bC$, then in analogy with  \eqref{gjfglkjglkjgergergereg} we have a functor  
\begin{equation}\label{ogjeriogjoregregeg}
C^{(-)} \colon \Sp^{\op}\to  \bC\, ,\quad  W\mapsto C^{W}\, .
\end{equation}
For a space $A$ we have the natural equivalence $C^{A}\simeq C^{\Sigma^{\infty}_{+}A}$.

We now construct the pairing.
We first observe that we have an equivalence of functors
$$\tilde Q_{C}\simeq C^{\tilde Q^{\hlg}} \colon (\BC^{\cC,\cB})^{\op}\to \bC\, .$$
Indeed, for $(X,U,B)$ in $\BC^{\cC,\cB}$ we have the natural equivalences
\begin{eqnarray*}\tilde Q_{C}(X,U,B)&\simeq &\Fib(C^{\iota P_{U}(X)}\to C^{\iota P_{U}(X\setminus B)}) \\&\simeq &\Fib( C^{  \Sigma^{\infty}_{+} \iota P_{U}(X)}\to C^{ \Sigma^{\infty}_{+}\iota  P_{U}(X\setminus B)})\\&\simeq &
C^{\Cofib({\Sigma^{\infty}_{+} \iota P_{U}(X\setminus B)\to \Sigma^{\infty}_{+} \iota  P_{U}(X)})}\, .
\end{eqnarray*}
We now form the left Kan extension of this equivalence along the functor
$$(\BC^{\cC,\cB})^{\op}\to (\BC^{\cC})^{\op}$$ and get  the natural transformation
$$\hat Q_{C}\simeq LK(C^{\tilde Q^{\hlg}})\stackrel{!}{\to} C^{RK({\tilde Q^{\hlg}})}\simeq C^{\hat Q^{\hlg}}\, .$$ 
Here $LK$ and $RK$ stand for the left, resp.\ right Kan extension. In general, the marked transformation
is not an equivalence since the functor \eqref{ogjeriogjoregregeg} in general does not preserve colimits.
 We now form the right Kan extension of this morphism along the functor
 $$(\BC^{\cC})^{\op}\to \BC^{\op}$$
and get the  morphism
\[p \colon Q_{C}\simeq RK(\hat Q_{C})\to RK(C^{\hat Q^{\hlg}})\stackrel{!}{\simeq}C^{LK(\hat Q^{\hlg}})\simeq C^{Q^{\hlg}}\simeq D_{C}(Q^{\hlg})\]
which is the desired pairing. Note that here the marked morphism is an equivalence since \eqref{ogjeriogjoregregeg} preserves limits. 

\section{The dualizing spectrum of a group}
\label{regoiegpergergreg}

 \subsection{Poincar{\'e} duality groups}
 
Let $G$ be a group in $\Spc$, or equivalently, a group-like  $E_{1}$-monoid.
Then we can form its classifying space $BG$ which is a pointed object of  $\Spc $. In fact,   every connected
pointed space $X$ in $\Spc$ is equivalent to  the classifying space  $B(\Omega X) $ of its loop space   $\Omega X$ considered as a group in $\Spc$. 

Considering $BG$ as an $\infty$-groupoid we define the category of $G$-spectra by 
\begin{equation}\label{qewfqewdqwd} G\Sp \coloneqq \Fun(BG,\Sp)\, . 
\end{equation} 
For $E$ in $G\Sp$ we can form the fixed points and orbits
$$E^{G} \coloneqq \lim_{BG}E\, , \quad E_{G} \coloneqq \colim_{BG}E$$ in $\Sp$.

We can form the spherical group ring $S[G] \coloneqq S\wedge G_{+}$ in $(G\times G)\Sp$ using the left and right actions {of} $G$ on itself. Following Klein \cite{klein} we then define the dualizing  spectrum of $G$ by 
$$D_{G} \coloneqq \lim_{BG}S[G]$$
in $G\Sp$, where we use the right action to form the fixed points and keep the residual left action. 
We will write $\underline{ D}_{G}$ in $\Sp$ for the underlying spectrum of $D_{G}$. 

For any  $E$ in $G\Sp$   Klein \cite[Sec.\,3]{klein} introduces a norm map {in $\Sp$}
 $$N^{E} \colon D_{G}\wedge_{G} E\to E^{G}
   $$as  the composition
   \begin{eqnarray}
D_{G}\wedge_{G} E&\stackrel{\deff}{=} &\colim_{BG}  (\lim_{BG} S[G]\wedge  E)\label{gwregwrefwrefwf}\\
 &\to &   \colim_{BG}  \lim_{BG}  (S[G]\wedge  E) \nonumber \\&\to& \lim_{BG}  \colim_{BG}  (S[G]\wedge  E) \nonumber\\&\stackrel{!}{\simeq} & \lim_{BG} E \nonumber\\&\simeq& E^{G}\, . \nonumber
\end{eqnarray}
The marked equivalence uses  $\colim_{BG}(S[G]\wedge  E)\simeq E$ in $G\Sp$ which holds since $S[G]\wedge  E$ is freely induced.
 If $BG$ is a compact object in $\Spc$, using that   $\Sp$ is stable, the limit over $BG$ commutes with $-\wedge E$
and the colimit of $BG$ so that the two arrows in \eqref{gwregwrefwrefwf} are equivalences.
In this case the norm map is an equivalence, compare \cite[Thm.\,D]{klein}.

Recall that homology of $BG$ with coefficients in $D_{G}\wedge E$
and the  cohomology of $BG$ with coefficients in $E$ are defined by
  $$H_{*}(BG,D_G \wedge E) \coloneqq \pi_{*}(D_G\wedge_{G}E) \quad \text{and} \quad H^{*}(BG,E) \coloneqq \pi_{-*}(E^{G})\, .$$
  The norm map induces a map from the homology of $BG$ with coefficients in $D_{G}\wedge E$ to the cohomology with
coefficients in $E$. If the underlying spectrum $\underline{D}_{G}$ is equivalent to a shift $S^{-n}$ of the sphere spectrum and $BG$ is compact in {$\Spc$}, 
then the norm map induces a  Poincar\'e duality isomorphism
$$H_{n-*}(BG,L\wedge E)\xrightarrow{\cong} H^{*}(BG,E)\, ,$$
where $L \coloneqq \Sigma^{n} D_{G}$.
Indeed we have:
\begin{theorem}[{Klein \cite[Thm.\,A]{klein}}]
Assume that $BG$ is compact in $\Spc$. Then the following are equivalent:
\begin{enumerate}
\item $BG$ is a Poincar\'e duality space.
\item $\underline{D}_{G}$ is a shift of the sphere spectrum.
\item $\underline{D}_{G}$ is a finite spectrum. 
\end{enumerate}
\end{theorem}

We now restrict to discrete groups $G$. The importance of the homotopy type  $\underline{D}_{G}$  indicated above
raises the question how it can be calculated and how it depends on the group. Klein \cite[Conj.~on Page~455]{klein} 
conjectured that $\underline{D}_{G}$ only depends on the quasi-isometry class of $G$ with respect to the word metric.

Referring to Section \ref{f2h23iorj2r32r32r2}, we consider the case $\bC=\Sp$ and the sphere spectrum $C \coloneqq S$. By Definition \ref{wrowpt32432424234234} we get a coarse 
cohomology theory
$$Q_{S} \colon \BC^{\op}\to \Sp\, .$$
  By $G_{can,min}$  (see \cite[Ex.\,2.21]{buen}) we denote  the object of  $\BC$    obtained by equipping the   group $G$  with the bornological coarse structure given by the minimal bornology  and the canonical coarse structure.
Our main technical result is now:
\begin{theorem}\label{ergoiergregege}
If $G$ is  {finitely generated and} torsion-free, then we have an equivalence
$\underline{D}_{G}\simeq Q_{S}(G_{can,min})$ in $\Sp$. 
\end{theorem}

%
%
%
%

Thus $\underline{D}_{G}$ is the value of a coarse cohomology theory on $G_{can,min}$  {and therefore we get as} an immediate consequence: 

\begin{kor} \label{werijgowergwerf}
\label{rwfgiowjgogrgergeg}  For  finitely generated and  torsion-free groups $G$
the spectrum $\underline{D}_{G}$ only depends on the coarse motivic spectrum $\Yo^{s}(G_{can,min})$.
In particular, it is an invariant of the quasi-isometry class.
\end{kor}

\begin{rem}
If $BG$ is compact in $\Spc$, then $G$ is finitely generated and torsion-free.
\end{rem}

\begin{ex}\label{weriogwerfwefwerfw}
If $M$ is a  complete, simply connected, negatively curved $n$-dimensional Riemannian manifold $M$, then its  Gromov boundary $ \partial   M$ is homeomorphic to $S^{n-1}$. By 
Higson--Roe \cite[Sec.~8]{hr}
we have a coarse homotopy equivalence between $M$ and the open cone over    $  \partial   M$. 
 By \cite[Sec.\,6]{ass} 
 we know the open cone and the Euclidean cone yield equivalent objects in $\Sp\cX$ after applying $\Yo^{s}$.
Since the Euclidean cone over $S^{n-1}$ is coarsely equivalent to $\R^{n}$ we get 
  $\Yo^{s}(M)\simeq \Yo^{s}(\R^{n})$.

If
 $G$ acts freely and cocompactly  on $M$, then we have coarse equivalence $G_{can,min}\to M$
 given by $g\mapsto g m_{0}$ for any choice of a point $m_{0}$ in $M$. 
 So   we get an equivalence
 $$\Yo^{s}(G_{can,min})\simeq  \Yo^{s}(\R^{n})\stackrel{!}{\simeq} \Sigma^{n}\Yo^{s}(*)\, ,$$ where $!$ is a standard calculation in coarse homotopy theory (see e.g.\,\cite[Ex.~4.9]{buen}).
%
Consequently,  $$Q_{S}(G_{can,min})\simeq Q_{S}(\R^{n})\stackrel{}{\simeq} \Sigma^{-n}Q_{S}(*)\simeq S^{-n}$$ as expected since $BG\simeq M/G$
is an $n$-dimensional Poincar\'e duality space.  \end{ex}

\begin{ex}\label{ex2345treew}
In the following example we provide a pair of groups $G$ and $H$ which are not quasi-isometric, but for which $\Yo^{s}(G_{can,min})\simeq \Yo^{s}(H_{can,min})$. This example shows that the first assertion of Corollary \ref{werijgowergwerf} is strictly stronger than the second.
We consider torsion-free and cocompact lattices $G$ in $SO(2n,1)$ and $H$ in $SU(n,1)$. Such lattices exist by a result of Borel \cite{MR0146301}.
 Then $G$ is quasi-isometric to the hyperbolic space $\mathbb{H}^{2n}$ and $H$ is quasi-isometric to the complex hyperbolic space $\mathbb{H}\C^{n}$ (of real dimension $2n$). 
By Mostow rigidity (Mostow \cite{MR0385004}, Kleiner--Leeb \cite[Cor.~1.1.4]{MR1608566}) $\mathbb{H}^{2n}$ and $\mathbb{H}\C^{n}$ are not quasi-isometric, and hence $G$ and $H$ are not quasi-isometric.
 
The boundaries of the  negatively curved spaces  $\mathbb{H}^{2n}$ and  $\mathbb{H}\C^{n}$ are both homeomorphic to $S^{2n-1}$.  Hence $$\Yo^{s}(G_{can,min})\simeq \Sigma^{2n-1}\Yo^{s}(*)\simeq \Yo^{s}(H_{can,min})$$ by Example \ref{weriogwerfwefwerfw}. \end{ex}

\subsection{{Proof of Theorem \ref{ergoiergregege}}}

Let $\iota \colon \Top\to  \Spc$ denote the canonical functor from topological spaces to {the $\infty$-category of} spaces.
For a group $G$ we denote by $\BG$ the category consisting of one object whose monoid of endomorphisms is given by the group $G$. Note that 
\begin{equation}\label{qefwfwefqwed} 
\iota |\Nerve(\BG)|\simeq BG \, ,
\end{equation}
where 
$|\Nerve(\BG)|$ is the geometric realization of the nerve of $\BG$.

By $\Orb(G)$ we denote the orbit category of $G$ which is the category of transitive $G$-sets and equivariant maps. We form the categories $$G\Top \coloneqq \Fun(\BG,\Top)$$ of $G$-topological spaces (i.e., objects are topological spaces with an action of $G$, and morphisms are equivariant continuous maps) and   
$$G\Spc \coloneqq \Fun(\BG,\Spc)\, , \quad G[\Spc] \coloneqq \Fun(\Orb(G)^{\op},\Spc)$$ of spaces with a $G$-action and {of} $G$-spaces. 
   The category $G[\Spc]$ models the $G$-equivariant homotopy theory and is the natural home for classifying spaces $E_{\cF}G$ for families $\cF$ of subgroups of $G$.
We have a functor  
$\iota_{G} \colon G\Top\to G[\Spc]$ which sends the $G$-topological space
$X$ to the functor
\begin{equation}
\label{eqer54gtrb}
\Orb(G)^{\op}\ni O\mapsto \iota_{G}(X)(O) \coloneqq \iota  \Map_{G}(O,X) \in \Spc\, ,
\end{equation}
where $ \Map_{G}(Y,X)$ denotes the topological space of $G$-equivariant maps from $Y$ to $X$ with the compact-open topology, and $O$ is considered as a discrete $G$-topological space.
   
The category $G\Spc$ models the homotopy theory of topological  spaces with $G$-action and equivariant maps, where weak equivalences are maps which are weak equivalences after forgetting the $G$-action. 
In contrast, by Elmendorfs theorem $\iota_{G}$ models the Dwyer--Kan localization of $G\Top$ at the equivariant weak equivalences.
  We have an isomorphism of monoids $$\End_{\Orb(G)}(G)\cong G^{\op}\, ,$$ and therefore an inclusion
  \begin{equation}\label{foiuehfiuweofewf}
\BG^{\op}\to \Orb(G)\, .
\end{equation}
This inclusion induces an  adjunction
\begin{equation}\label{whgwujvwwef}
\Res :G[\Spc]\leftrightarrows G\Spc:\Coind
\end{equation} relating the two categories.

Furthermore, we let
$$G\Sp \coloneqq \Fun(\BG,\Sp)\, , \quad  G[\Sp] \coloneqq \Fun(\Orb(G)^{\op},\Sp)$$
be the categories of spectra with a $G$-action and of naive $G$-spectra. Since $G$ is discrete and in view of \eqref{qefwfwefqwed} this new definition of $G\Sp$ is equivalent to the former \eqref{qewfqewdqwd}.
  The inclusion \eqref{foiuehfiuweofewf} induces an  adjunction
\begin{equation}\label{bglkjo4bg45g}
\Res :G[\Sp]\leftrightarrows G\Sp:\Coind\, .
\end{equation}

An $\Omega$ spectrum  is a spectrum $(E_{n},\sigma_{n})_{n\in \nat}$ in  topological spaces  such that $\sigma_{n} \colon E_{n}\to \Omega E_{n+1}$ is a weak equivalence. A weak equivalence between $\Omega$-spectra is a morphism which is a level-wise weak equivalence. We denote by $\Sp^{\Omega}$ the ordinary category  of $\Omega$-spectra. The relative category $(\Sp^{\Omega},W)$, where $W$ denotes the class of weak equivalences, is a presentation of the category of spectra. In particular, we have a functor
$$\kappa \colon \Sp^{\Omega}\to \Sp^{\Omega}[W^{-1}]\simeq \Sp\, .$$
 
We furthermore  consider the category  $$G\Sp^{\Omega} \coloneqq \Fun(\BG,\Sp^{\Omega})$$ of $\Omega$-spectra with a $G$-action and use {the} symbol $\kappa$ also for the induced functor
$$\kappa \colon G\Sp^{\Omega}\to  G\Sp\, .$$

We consider the $G$-spectrum $S[G]$ in $G\Sp$, i.e., we only keep the right-action of $G$ on itself which is used to form the limit over $\BG$ later.
  
\begin{rem}\label{efowefpfr23}
In greater detail, $S[G]$ is given by $\coprod_{g\in G} S$,
where the $G$-action is given by the action of $G$ on the index set by right multiplication.
The technical description is
$$S[G] \coloneqq \Ind_{1}^{G}(S)\, ,$$
where $\Ind_{1}^{G} \colon \Sp\to G\Sp$ is the left-adjoint of the
forgetful functor
$G\Sp\to \Sp$.

Equivalently, we can choose an $\Omega$-spectrum $QS$  in $\Sp^{\Omega}$  with $\kappa(QS)\simeq S$.
Then we form the $G$-$\Omega$-spectrum $\Map_{c}(G,QS)$  of compactly supported maps from $G$ to $QS$ (see below for details), where
$G$ is considered as a discrete $G$-space with the right action.
Then we have an equivalence
\begin{equation}\label{rgerergergergreg}
S[G]\simeq \kappa  \Map_{c}(G,QS) \, .
\end{equation}
In order to see this equivalence we first observe that for a finite discrete space $F$ we have an equivalence 
$$S[F]\simeq \coprod_{F} S\simeq \prod_{F} S \simeq \prod_{F} \kappa {(QS)}\simeq \kappa \prod_{F}QS\simeq \kappa \Map(F,QS)\, .$$
This equivalence is functorial for embeddings of finite sets where on the side of the mapping spaces we use extension of maps by zero. 
We then use that
$$ S[G]\simeq \colim_{F\subseteq G} S[F]\simeq  \colim_{F\subseteq G} \kappa \Map(F,QS)\simeq 
\kappa \colim_{F\subseteq G} \Map(F,QS) \simeq \kappa \Map_{c}(G,QS)\, , $$
where the colimit is taken over the finite subsets of $G$, and 
where for the last equivalence we use that extension by zero of functions is level-wise a closed embedding of $\Omega$-spectra.\footnote{We thank Thomas Nikolaus for pointing out that the equivalence \eqref{rgerergergergreg} should be justified.}
\end{rem}

We have a functor $$\lim_{\BG} \colon G\Sp\to \Sp\, ,$$
and, by definition, an equivalence $$\underline{D}_{G}\simeq \lim_{\BG} S[G]\, .$$


Let $X$ be a $G$-topological space and $Z$ a pointed topological space.
For  a subset $K$ of $X$ we let $\Map_{ K}(X
,Z)$ denote the subspace of $\Map (X,Z)$ of maps which send $X\setminus K$ to the base point.
We define the $G$-subset of $\Map (X,Z)$
$$\Map_{ c}(X,Z) \coloneqq \bigcup_{K}\Map_{ K}(X,Z)$$
of compactly supported maps, where $K$ runs over all compact subsets of $X$.
We equip $\Map_{ c}(X,Z)$ with the inductive limit topology. Note that this topology  is  {in general} finer than the induced topology from $\Map (X,Z)$.

Let $X,Y$ be  $G$-topological spaces and $Z$ be a pointed topological space.
\begin{lem}\label{ewfiuwehfiuewfewfewf}
If $G$ acts properly and cocompactly on $X$ and $Y$, then we have a homeomorphism
\begin{equation}
\label{hrthrh45z}
\Map_{G}(X,\Map_{c}(Y,Z)) \cong \Map_{G}(Y,\Map_{c}(X,Z))\, .
\end{equation}
\end{lem}

\begin{proof}
We define the $G$-space
$$\Map_{d}(X\times Y,Z) \coloneqq \colim_{(K,L)} \Map_{G(K\times L)}(X\times Y,Z)$$
equipped with the inductive limit topology, where $K$ (or $L$) runs over the compact subsets of $X$ (resp.~of $Y$)  {and $G$ acts diagonally on $X \times Y$.}
We compare both sides of \eqref{hrthrh45z} with $\Map_{d}(X\times Y,Z)^{G}$.  {We carry out the arguments only for the case of $\Map_{G}(X,\Map_{c}(Y,Z)) $ since the other case is completely analogous.}


Assume that $f$ belongs to  $\Map_{G}(X,\Map_{c}(Y,Z))$.
By the exponential law for maps between sets it corresponds to a $G$-equivariant map $\tilde f \colon X\times Y\to Z$ which by $G$-{equi}variance is determined by its restriction to
$K\times Y$ for any compact subset $K$ of $X$ with $GK=X$. Since we equip
$\Map_{c}(Y,Z)$ with the inductive limit topology, there exists a compact subset $L$ of $Y$ with
$\tilde f(k,y)=*$ for all $k\in K$ and $y\in Y\setminus L$. In other words,
$\tilde f\in \Map_{G(K\times L)}(X\times Y,Z)^{G}$.
In this way we define a map
$$ \Map_{G}(X,\Map_{c}(Y,Z))\to  \Map_{d}(X\times Y,Z)^{G}\, .$$

Assume now that
$\tilde f $ belongs to  $\Map_{d}(X\times Y,Z)^{G}$.
Then there is a pair $(K,L)$ of compact subsets of $X$ and $Y$, respectively, such that $\tilde f$ is supported on $G(K\times L)$.
Since $G$ acts properly on $X$, the set
$F \coloneqq \{g \in G\:|\: gK\cap K\not=\emptyset\}$ is finite. Then $\tilde L \coloneqq FL$ is compact and
$\tilde f(x,y)=*$ for $x$ in $K$ and $y\in Y\setminus \tilde L$. Let
$f \colon X\to \Map(Y,Z)$ be the adjoint of $\tilde f$. Then $f_{|K}$ takes values in $\Map_{\tilde L}(Y,Z)$. This shows  that $f\in \Map_{G}(X,\Map_{c}(Y,Z))$.
In this way we have constructed the inverse map
\[\Map_{d}(X\times Y,Z)^{G}\to  \Map_{G}(X,\Map_{c}(Y,Z))\, .\qedhere\]
\end{proof}

If $E=(E_{n},\sigma_{n})_{n\in \nat}$ is a $G$-$\Omega$-spectrum, then for a $G$-topological space $X$  we get a $G$-$\Omega$-spectrum 
$$\Map(X,E) \coloneqq (\Map(X,E_{n}), \sigma^{X}_{n})_{n\in \nat}\, ,$$
where 
$$\sigma^{X}_{n} \colon \Map(X,E_{n})\xrightarrow{\sigma_{n}} \Map(X,\Omega E_{n+1})\cong \Omega \Map(X,E_{n+1})\, .$$
If $X$ is a CW-complex and $E$ is an $\Omega$-spectrum, then we have the equivalence in $\Sp$
\begin{equation}\label{wfroihiofefwefwefewf}
\kappa  \Map(X,E) \simeq (\kappa E)^{ \iota X} \, .
\end{equation}
In order to see this note that both sides are cohomology theories in the CW-complex argument and coincide for $X=*$.

Furthermore, if $X$ is a free $G$-CW-complex and $E$ is a $G$-$\Omega$-spectrum, then we have the equivalence in $\Sp$
\begin{equation}\label{hfiuwehfiuewfhew}
\kappa  \Map_{G}(X,E) \coloneqq \kappa \lim_{\BG} \Map(X,E)\simeq \lim_{\BG}\left[(\kappa E)^{\Res {(\iota_{G} X)}}\right]\,,
\end{equation}
where   $\Res$ is as in \eqref{whgwujvwwef} {and $\iota_G$ as in \eqref{eqer54gtrb}.}
Again, both sides are cohomology theories on free $G$-CW-complexes and coincide on $X=G$.

Similarly, we define the $G$-$\Omega$-spectrum 
$ \Map_{c}(X,E)$ by $({\Map_c}(X,E_{n}), \sigma_{c,n}^{X} )_{n\in \nat}$,
where 
$$\sigma_{c,n}^{X} \colon \Map_{c}(X,E_{n})\xrightarrow{\sigma_{n}} \Map_{c}(X,\Omega E_{n+1})\cong \Omega \Map_{c}(X,E_{n+1})\, .$$
For the last isomorphism  we used  that the circle $S^{1}$ is compact.

We now choose an $\Omega$-spectrum $QS$ representing the sphere spectrum. We consider $G$ as a discrete $G$-space and note that $G$  acts properly and cocompactly on $G$.
By  Remark \ref{efowefpfr23} we have   
$$S[G]\simeq \kappa  \Map_{c}(G,QS)\, .$$
The Rips complex of a $G$-coarse space $X$ with coarse structure $\cC$ is defined by  
$$\Rips(X) \coloneqq \colim_{U\in \cC^{G}}  P_{U}(X)\, ,$$ where
the colimit is interpreted in the category $G\Top$ of $G$-topological spaces. Then by \cite[Lem.\,11.4]{equicoarse} we have an equivalence $$\iota_{G} \Rips(G_{can})  \simeq E_{\Fin}G$$
in $G[\Spc]$.  
Since we assume that $G$ is torsion-free we have an equivalence $E_{\Fin}G\simeq EG$.

Since $G$ is finitely generated the coarse structure $\cC$ of $G_{can}$ is generated by a single invariant entourage $U_{gen}$. Hence we have an equivariant homeomorphism
$$\Rips(G_{can})\cong \colim_{n\in \nat} P_{U_{gen}^{n}}(G )\, .$$

We further observe that $P_{U_{gen}^{n}}( {G})$ is a  {locally finite} $G$-CW-complex, and that the morphisms
$P_{U_{gen}^{n}}( {G})\to P_{U^{n+1}_{gen}}( {G})$   are inclusions of subcomplexes. It follows that the colimit over these inclusions is a homotopy colimit, i.e., that we have the equivalence
$$EG\simeq \iota_{G}  \Rips(G_{can})  \simeq \colim_{n\in \nat} \iota_{G}  P_{U_{gen}^{n}}(G)$$
in $G[\Spc]$. Let
$$(-)^{G} \colon G[\Sp]\to {\Sp}$$
be the evaluation functor at the one-point $G$-set. For a $G$-$\Omega$-spectrum $E$ we have the following chain of equivalences in $\Sp$:
\begin{eqnarray*}
\lim_{\BG} \kappa E&\simeq& (\Coind \ \kappa E)^{G}\\
&\simeq & ((\Coind \ \kappa E)^{EG})^{G}\\
&\simeq&   ((\Coind \ \kappa E)^{  \colim_{n\in \nat} \iota_{G} P_{U_{gen}^{n}}(G)})^{G}\\
&\simeq& \lim_{n\in \nat}  ((\Coind \ \kappa E)^{    \iota_{G} P_{U_{gen}^{n}}(G)})^{G}\\
&\simeq& \lim_{n\in \nat}\lim_{BG}(\kappa E)^{\Res (\iota_{G} P_{U_{gen}^{n}}(G))}\\
&\stackrel{\eqref{hfiuwehfiuewfhew}}{\simeq}& \lim_{n\in \nat} \kappa  \Map_{G}(P_{U_{gen}^{n}}(G ), E)\, .
\end{eqnarray*}
For the first equivalence we interpret $\Coind$ in \eqref{bglkjo4bg45g} as the right Kan extension functor along $\BG\to G\Orb^{\op}$ and use the point-wise formula for the evaluation of the result at the initial object of $G\Orb^{\op}$.  For the  last equivalence we use  the fact that $P_{U_{gen}^{n}}(G )$ is a free $G$-CW-complex since $G$ is torsion-free.

We obtain the equivalence
$$\underline{D}_{G}\simeq \lim_{n\in \nat} {\kappa \Map_{G}(P_{U_{gen}^{n}}(G ),\Map_{c}(G,QS))}\, .$$
Since $G$ acts properly and cocompactly on both $P_{U_{gen}^{n}}(G)$ and $G$, we get by Lemma \ref{ewfiuwehfiuewfewfewf}
\[\underline{D}_{G} \simeq \lim_{n\in \nat}{\kappa \Map_{G}(G,\Map_{c}(P_{U^{n}_{gen}}(G ),QS))  } \simeq \lim_{n\in \nat} \kappa \Map_{c}(P_{U^{n}_{gen}}(G),QS)\, ,\]
where the second equivalence is induced by the evaluation at the identity of $G$.
 We now observe that the subsets of the form
 $\overline{P_{U^{n}_{gen}}(G )\setminus P_{U^{n}_{gen}}(G\setminus B)}$ for all bounded subsets $B$ of $G_{can,min}$ (i.e., finite subsets of $G$)  are cofinal in the set of compact subsets of $P_{ {U^{n}_{gen}}}(G)$. 
 It follows that
 $$\Map_{c}(P_{U^{n}_{gen}}(G),QS) \cong \colim_{B\in \cB} \Fib\big(\Map (P_{U^{n}_{gen}}(G),QS)\to \Map(P_{U^{n}_{gen}}(G\setminus B),QS)\big)\, ,$$
 where $\cB$ denotes the bornology of $G_{can,min}$.
Hence we get
$$\underline{D}_{G}\simeq  \lim_{n\in \nat}   \kappa   \colim_{B\in \cB} \Fib\big(\Map (P_{U^{n}_{gen}}(G),QS)\to \Map(P_{U^{n}_{gen}}(G\setminus B),QS)\big)\,.$$
Let $QS_{k}$ be the $k$'th space of the $\Omega$-spectrum $QS$. For $B$ in $\cB$ we   set
$$F(B)_{k} \coloneqq \Fib\big(\Map (P_{U^{n}_{gen}}(G),QS_{k})\to \Map(P_{U^{n}_{gen}}(G\setminus B),QS_{k})\big)\, .$$ The structure maps 
 of the spectrum $QS$ induce structure maps  $$\sigma(B)_{n} \colon \Sigma F(B)_{k}\to F(B)_{k+1}$$ which turn    $F(B) \coloneqq (F(B)_{k},\sigma(B)_{n})_{k\in \nat}$ into an $\Omega$-spectrum.
 If $B^{\prime}$ is a second element in $\cB$ such that $B\subseteq B^{\prime}$, then the natural morphism
 $F(B)_{k}\to F(B^{\prime})_{k}$ is a closed embedding. Indeed, it is the embedding of the space of $QS_{k}$-valued functions on $P_{U^{n}_{gen}}(G)$ which are constant on the subset $P_{U^{n}_{gen}}(G\setminus B)$ into the space of such functions which are constant on the smaller subset $P_{U^{n}_{gen}}(G\setminus B^{\prime})$. Since taking homotopy groups commutes with 
 filtered colimits of systems of  pointed spaces whose   structure maps are closed embeddings
 we can conclude that $\colim_{B\in \cB} F(B) $ is again an $\Omega$-spectrum, and that this also represents the spectrum
 $\colim_{B\in \cB} \kappa F(B)$.
Hence we can switch the order of $\kappa$ and taking the colimit. Furthermore, because
$P_{U^{n}_{gen}}(G\setminus B)\to P_{U^{n}_{gen}}(G)$ is an inclusion of a {locally finite} subcomplex, the induced map between $\Omega$-spectra is a fibration between $\Omega$-spectra. Therefore $\Fib$ in the formula above can be commuted with $\kappa$. Therefore
$$\underline{D}_{G}\simeq  \lim_{n\in \nat}   \colim_{B\in \cB} \Fib\big(\kappa  \Map (P_{U^{n}_{gen}}(G),QS) \to \kappa \Map(P_{U^{n}_{gen}}(G\setminus B),QS) \big)\ .$$ 
We now use the relation \eqref{wfroihiofefwefwefewf} in order to get the equivalence
$$\underline{D}_{G}\simeq  \lim_{n\in \nat}   \colim_{B\in \cB} \Fib\big( (\kappa QS)^{\iota P_{U^{n}_{gen}}(G)}\to  (\kappa QS)^{\iota P_{U^{n}_{gen}}(G\setminus B)}\big)\, .$$ 
 
 Finally, by cofinality we replace the limit over $\nat$ by the limit over $\cC$. In view of \eqref{ewfweojo2p4jrpwfewrfw23fffwefwefwef} we get the desired equivalence 
 $$\underline{D}_{G}\simeq    Q_{S}(G_{can,min})\, ,$$
which finishes the proof of Theorem~\ref{ergoiergregege}.

\begin{rem}
Let $\Orb_{\cF}(G)$ denote the full subcategory of the orbit category of $G$  of transitive $G$-sets with stabilizers in the family of subgroups $\cF$. We set $$G_{\cF}[\Sp] \coloneqq \Fun(\Orb_{\cF}(G)^{\op},\Sp)\, .$$  
Then for every two families $\cF$ and $\cF^{\prime}$ with $\cF\subseteq \cF^{\prime}$ we have a {corresponding} pair of adjoint  functors $(\Ind_{\cF}^{\cF^{\prime}},\Res^{\cF^{\prime}}_{\cF})$,  see \cite[Sec.\,10.3]{equicoarse}. If  $E\in G_{\All}[\Sp]$, then we define
$$E^{(h_{\cF}G)} \coloneqq \lim_{\Orb_{\cF}(G)} \Res^{\All}_{\cF} E\, .$$
If $H$ is a subgroup of $G$, then we have an induction functor
$$\Ind_{H,\cF}^{G} \colon H_{\cF\cap H}[\Sp]\to G_{\cF}[\Sp]\, .$$
We could consider $S[G]$ as an object $\Ind_{ \{1\},\All}^{G}( S )$ of $G_{\All}[\Sp]$. Then by construction
\[\underline{D}_{G}\simeq S[G]^{(h_{\{1\}}G)}\, .\qedhere\]
\end{rem}

An appropriate modification (with $\BG\simeq \Orb_{\{1\}}(G)$ replaced by $\Orb_{\Fin}(G)$) of the proof of Proposition \ref{ergoiergregege} shows  {the following more general statement:}
\begin{prop}\label{efoiwpfwefwefewfwe}
For every finitely generated group $G$ we have an equivalence
$$S[G]^{(h_{\Fin}G)}\simeq Q_{S}(G_{can,min})\, .$$
\end{prop}
If $G$ is torsion-free, then we have the equality of families $\Fin=\{1\}$ and Proposition \ref{ergoiergregege} follows from
Proposition \ref{efoiwpfwefwefewfwe}.

\bibliographystyle{amsalpha}
\bibliography{coho}

\providecommand{\bysame}{\leavevmode\hbox to3em{\hrulefill}\thinspace}
\providecommand{\MR}{\relax\ifhmode\unskip\space\fi MR }
\providecommand{\MRhref}[2]{%
  \href{http://www.ams.org/mathscinet-getitem?mr=#1}{#2}
}
\providecommand{\href}[2]{#2}
\begin{thebibliography}{BEKW20}

\bibitem[BE20a]{ass}
U.~Bunke and A.~Engel, \emph{{Coarse assembly maps}}, J.\ Noncommut.\ Geom.
  \textbf{14} (2020), no.~4, 1245--1303.

\bibitem[BE20b]{buen}
\bysame, \emph{{Homotopy theory with bornological coarse spaces}}, Lecture
  Notes in Mathematics, vol. 2269, Springer, 2020.

\bibitem[BEKW20]{equicoarse}
U.~Bunke, A.~Engel, D.~Kasprowski, and Ch. Winges, \emph{Equivariant coarse
  homotopy theory and coarse algebraic {$K$}-homology}, {$K$}-theory in
  algebra, analysis and topology, Contemp. Math., vol. 749, Amer. Math. Soc.,
  2020, pp.~13--104.

\bibitem[Bor63]{MR0146301}
A.~Borel, \emph{{Compact Clifford--Klein forms of symmetric spaces}}, Topology
  \textbf{2} (1963), 111--122.

\bibitem[Bro82]{brown}
K.~S. Brown, \emph{{Cohomology of Groups}}, Graduate Texts in Mathematics,
  vol.~87, Springer, 1982.

\bibitem[FO16]{Fukaya:2013aa}
T.~Fukaya and Sh.-I. Oguni, \emph{{Coronae of relatively hyperbolic groups and
  coarse cohomologies}}, J. Topol. Anal. \textbf{8} (2016), no.~3, 431--474.

\bibitem[Ger95]{gersten}
S.~M Gersten, \emph{Isoperimetric functions of groups and exotic cohomology},
  Combinatorial and geometric group theory (Edinburgh, 1993), LMS Lecture
  Notes, vol. 204, 1995, pp.~87--104.

\bibitem[Har20]{hartmann}
E.~Hartmann, \emph{Coarse cohomology with twisted coefficients}, Math. Slovaca
  \textbf{70} (2020), no.~6, 1413--1444.

\bibitem[HR95]{hr}
N.~Higson and J.~Roe, \emph{{On the coarse Baum--Connes conjecture}}, {Novikov
  conjectures, index theorems and rigidity, Vol.~2}, London Mathematical
  Society Lecture Notes 227, Cambridge University Press, 1995.

\bibitem[KL97]{MR1608566}
B.~Kleiner and B.~Leeb, \emph{{Rigidity of quasi-isometries for symmetric
  spaces and Euclidean buildings}}, Pub. math. l'I.H.{\'E}.S. \textbf{86}
  (1997), 115--197.

\bibitem[Kle01]{klein}
J.~R. Klein, \emph{{The dualizing spectrum of a topological group}}, Math. Ann.
  \textbf{319} (2001), no.~3, 421--456. \MR{1819876}

\bibitem[Mos73]{MR0385004}
G.~D. Mostow, \emph{{Strong Rigidity of Locally Symmetric Spaces}}, Princeton
  University Press, 1973, Annals of Mathematics Studies, No.~78.

\bibitem[Roe93]{roe_coarse_cohomology}
J.~Roe, \emph{{Coarse Cohomology and Index Theory on Complete Riemannian
  Manifolds}}, {Memoirs of the Amer.\ Math.\ Soc.} \textbf{104} (1993),
  no.~497, 1--90.

\bibitem[Roe96]{roe_index_coarse}
\bysame, \emph{{I}ndex {T}heory, {C}oarse {G}eometry, and {T}opology of
  {M}anifolds}, {CBMS Regional Conference Series in Mathematics}, vol.~90, AMS,
  1996.

\bibitem[Sch99]{schmidt}
A.~Schmidt, \emph{{Coarse Geometry via Grothendieck Topologies}}, Math. Nachr.
  \textbf{203} (1999), 159--173.

\bibitem[Wul22]{Wulff:2020vh}
Ch. Wulff, \emph{{Equivariant Coarse (Co-)Homology Theories}}, SIGMA
  \textbf{18} (2022).

\end{thebibliography}

\end{document}